\documentclass[11pt]{article}
\usepackage{amsfonts}
\usepackage{amsfonts}
\usepackage{amsfonts}
\usepackage{amsfonts}
\usepackage{amsfonts}
\usepackage{amsfonts}
\usepackage{amsfonts}
\usepackage{amsfonts}
\usepackage{amsfonts}
\usepackage{amsfonts}
\usepackage{amsfonts}
\usepackage{amsfonts}
\usepackage{amsfonts}
\usepackage{amsfonts}
\usepackage{amsfonts}
\usepackage{amsfonts}
\usepackage{amsfonts}
\usepackage{mathrsfs}
\usepackage{amssymb,mathrsfs,amsmath,amsthm,amsfonts}
\usepackage{graphicx}
\usepackage[pdfstartview=FitH]{hyperref}
\usepackage[title]{appendix}

\providecommand{\U}[1]{\protect\rule{.1in}{.1in}}
\allowdisplaybreaks 

\numberwithin{equation}{section}
\newtheorem{thm}{Theorem}[section]

\newtheorem{proposition}{Proposition}[section]
\newtheorem{corollary}{Corollary}[section]
\newtheorem{lemma}{Lemma}[section]

\newtheorem{Rem}{Remark}[section]
\newtheorem{example}{Example}[section]
\newtheorem{hypothesis}{Hypothesis}[section]

\numberwithin{equation}{section}

\title{\bf{Malliavin Matrix of Degenerate SDE and Gradient Estimate}\thanks{The authors
were Supported by 973 Program, No. 2011CB808000 and
Key Laboratory of Random Complex Structures and Data Science, No.2008DP173182, NSFC, No.:10721101, 11271356, 11371041.}}

{\author{\ Dong Zhao$^{\dag}$,  \  Xuhui Peng$^{\ddag}$ \\
{\em\small  Academy of Mathematics and Systems Science, Chinese Academy of Sciences,}\\
{\em\small  Beijing, {\rm 100190}, P.R.China.}
 }}

\date{}
\begin{document}
\footnotetext{$^\dag$Email:  dzhao@amt.ac.cn}
\footnotetext{$^\ddag$Corresponding author.  Email: pengxuhui@amss.ac.cn}

\maketitle

\begin{abstract}
\noindent In this article, we  prove that the inverse of Malliavin
matrix belong to $L^p(\Omega,\mathbb{P})$ for a kind of degenerate
stochastic differential equation(SDE) under some  conditions, which
like to H\"{o}rmander condition, but  don't need all the
coefficients of the SDE are smooth. Furthermore, we obtain  a
locally uniform estimation for Malliavin matrix, a gradient
estimate, and prove that the semigroup generated by the SDE is
strong Feller. Also some examples are given.

\vskip0.5cm\noindent{\bf Keywords:}
Degenerate stochastic differential equation; Gradient estimate; Strong Feller; Malliavin calculus;
H\"{o}rmander condition.  \vspace{1mm}\\
\noindent{{\bf AMS 2000:} 60H10, 60H07.}
\end{abstract}
\section{Introduction and Notations}
In this article, we consider the following degenerate stochastic
differential equations(SDE)
\begin{equation}
\label{1-1}
\left\{
\begin{split}
&x_t=x+\int_0^ta_1(x_s,y_s)ds,\\
&y_t=y+\int_0^ta_2(x_s,y_s)ds+\int_0^tb(x_s,y_s)dW_s.
\end{split}\right.
\end{equation}
where $x \in \mathbb{R}^{m},~y \in \mathbb{R}^{n}, b\in
\mathbb{R}^{n\times d},~W_s$ is a $d$-dimensional standard Brownian
motion. Eq.$(\ref{1-1})$ is a  model  for many physical phenomenons.
For example,  $x_t$ represents the  position of an object and $y_t$
represents the momentum of the object. When a
random force  affects  the object, the momentum of the object changes  firstly, then that would
lead to  the variety of the object position. Thus the equation which describes the
movement of the object  is naturally degenerate as Eq.$(\ref{1-1})$. To understand the
long time behavior of  the movement of the object,  we need to study  the  ergodicity of
Eq.$(\ref{1-1})$. For this reason, the gradient estimate  of the
semigroup and the strongly Feller property associated to the
solution should be considered, and   the solution is  ergodic if one also
knows that the solution is topological irreducible and has an invariant
probability measure.

Let $\mathbb{P}_{x,y}$ be the law of the
solution to equation Eq.$(\ref{1-1})$ with initial value $(x,y)$, and
$P_t$ be the transition semigroup of Eq.$(\ref{1-1})$
\begin{align*}
P_tf(x,y):=\mathbb{E}_{x,y}f(x_t,y_t),~f\in \mathscr{B}_b(\mathbb{R}^{m}\times\mathbb{R}^{n};\mathbb{R}),
\end{align*}
where  $\mathscr{B}_b(\mathbb{R}^{m}\times\mathbb{R}^{n};\mathbb{R})$ denotes the collection of bounded
 Borel measurable functions and  $\mathscr{B}(\mathbb{R}^{m}\times\mathbb{R}^{n};\mathbb{R})$ denotes the
 collection of  Borel measurable functions.

For the general SDE
\begin{eqnarray}\label{1-4}
  X_t=x+\int_0^t V_0(X_s)ds+\sum_{j=1}^d\int_0^t  V_j(X_s)\circ dW_j(s), \ \ x\in \mathbb{R}^{m+n}.
\end{eqnarray}
 The   H\"{o}rmander condition  $(\mathbf{H})$  is that the vector
space spanned by the vector fields
\begin{eqnarray*} \label{1-2}
(\mathbf{H})\ \ \
V_1,\cdots,V_d,~[V_i,V_j],0 \leq i,j \leq d,~[[V_i,V_j],V_k],~0\leq i,j,k \leq d,~\cdots,
\end{eqnarray*}
at point $x$ is $\mathbb{R}^{m+n}$. The coefficients  are infinitely
differentiable functions with bounded partial derivatives of all
order. If  the  H\"{o}rmander condition  $(\mathbf{H})$ holds for
any  $x\in\mathbb{R}^{m+n}$ ,  the process  $X_t$ has a smooth
density  and  the transition semigroup of $(\ref{1-4})$ is strong
Feller (see \cite{Ichihara}, \cite{Kusuoka}, \cite{Nualart},
\cite{Shigekawa} etc).

Let  $V=(V_1,\cdots,V_d)$ and $P_t(x,\ )$ be the transition probabilities probabilities of the
 $X_t$ in $(\ref{1-4})$. When $ VV^*$, where $*$ means the transpose of the matrix, being uniformly elliptic,  the two-sided bounds of the
  density for  $P_t(x,\ )$  were  given   in \cite{Sheu} by using stochastic control tools.
  There also many other excellent works when $VV^*$ is non-degenerate.

There  is also many works in the hypoelliptic setting.   For the special case $V_0\equiv0,$ in \cite{Kusuoka},
Kusuoka  and Stroock  gave the two-sided bounds of the density for $P_t(x,\ )$ under
some conditions   which need some uniformity on  $V_1,\cdots,V_d.$
Recently, in \cite{Delarue},  Delarue and Menozzi  considered   the following SDE,
\begin{eqnarray}\label{1-5}
\left\{
  \begin{split}
X_t^1&=x_1+\int_0^t F_1(s,X_s^1,\cdots,X_s^n)ds+\int_0^tb(s,X_s^1,\cdots,X_s^n)dW_s,
\\
X_t^2&=x_2+\int_0^t F_2(s,X_s^1,\cdots,X_s^n)ds,
\\
X_t^3&=x_3+\int_0^t F_3(s,X_s^2,\cdots,X_s^n)ds,
\\
&  \ \ \vdots
\\
X_t^{n}&=x_n+\int_0^tF_n(s,X_s^{n-1},X_s^n)dt.
\end{split}
\right.
\end{eqnarray}
If the spectral of the $A(t,x)=[bb^*](t,x)$, is included in $[\Lambda^{-1}, \Lambda]$ for some $\Lambda\geq1$
and $D_{x_{i-1}}F_i(t,x_{i-1},x_i,\cdots,x_n)$ is non-degenerate, uniformly in space and time, they
gave the two-sided bounds of the density for to the solution to Eq.$(\ref{1-5})$.\
Another work is that  in \cite{Ciniti},
the authors    considered the SDE as
\begin{eqnarray}\label{1-6}
  X_t^i=x_i+W_t^i,\ \ \forall i \in [1,n], \ \ \ X_t^{n+1}=x_{n+1}+\int_{0}^t|X_s^{1,n}|^kds,
\end{eqnarray}
here $X_s^{1,n}=(X_s^1,\cdots,X_s^n)$ and    they gave  the two-sided bounds estimation for the transition function
$p(t,x,.)$ in \cite{Ciniti}.

There are also  many  other  researches  on the special case of Eq.$(\ref{1-1})$,  such as
 \cite{J.C. Mattingly}, \cite{Kliemann}, \cite{D.Talay} and so on.
In \cite{J.C. Mattingly} and  \cite{D.Talay}, the authors studied the ergodicity and
in \cite{Kliemann}, the author studied the  recurrence and invariant measure.

In most of the above works, the coefficients are smooth or some uniform conditions are needed. Since
our aim in this article is to prove the strong Feller property and give a  gradient estimate of the semigroup, we don't need the smooth conditions for all the coefficients or some uniform conditions. Instead of  the  H\"{o}rmander conditions, we give some  new conditions, which are  equivalent to the H\"{o}rmander condition if the coefficients are smooth, and proved that the inverse of the Malliavin matrix is $L^p$ integrable for any $p\ge 0.$ Furthermore, our new conditions  also    ensure   that we can obtain  a  gradient estimate
and the strong Feller property.

We haven't  obtain  the smoothness of the density or
the two-sided bounds of the density  for the  lack of  smoothness or some uniform conditions on the coefficients.

Before we give the organization of this article, we introduce  some notations.
For $j\in
\mathbb{N},$ let  $C^j(\mathbb{R}^m\times \mathbb{R}^n;
\mathbb{R}^{l})$ be the collection of functions which have  continuous
derivatives  up to order $j$ and $C_{b}^j(\mathbb{R}^m\times \mathbb{R}^n;
\mathbb{R}^{l})$ be the collection of functions in
$C^j(\mathbb{R}^m\times \mathbb{R}^n; \mathbb{R}^{l})$ with bounded derivatives. Sometimes,
we will use $C_{b}^{j}$ and $C^{j}$ instead of them  for the convenience of writing.
For
$l\in \mathbb{N},~k=(k_{1}(x,y),\cdots,k_{l}(x,y))^{* }\in C^1(\mathbb{R}^m\times
\mathbb{R}^n; \mathbb{R}^{l}),$ $x=(x_{1},\cdots,x_{m})^{* },~y=(y_{1},\cdots,y_{n})^{* }$,
\begin{eqnarray*}
\nabla _{x_{i}}k &=&\left( \frac{\partial k_{1}}{\partial x_{i}},\cdots
,\frac{\partial k_{l}}{\partial x_{i}}\right) ^{* },~i=1,\cdots ,m, \ \nabla _{x}k
= (\nabla _{x_{1}}k,\cdots ,\nabla _{x_{m}}k),\\
\nabla _{y_{j}}k&=&\left( \frac{\partial k_{1}}{\partial
y_{j}},\cdots ,
\frac{\partial k_{l}}{\partial y_{j}}\right) ^{* },~j=1,\cdots ,n,\ \nabla_{y}k
= (\nabla _{y_{1}}k,\cdots ,\nabla _{y_{n}}k),
\end{eqnarray*}
and $\nabla k=(\nabla_xk,\nabla_yk)$. If $a_1 \in C^{j_0}(\mathbb{R}^m\times
\mathbb{R}^n; \mathbb{R}^{m})$ for some $j_0 \in \mathbb{N}$, we
define vector fields:
\begin{eqnarray*}
\label{2-117}
&&\mathcal{A}_{1} =\big\{\nabla _{y_{j}}a_{1},~j=1,\cdots ,n\big\}, \\
\label{2-118}&&\mathcal{A}_{l} =\big\{\nabla _{y_{j}}k,~j=1,\cdots ,n,~-\nabla
_{x}a_{1}\cdot k+\nabla _{x}k\cdot a_{1}:k\in \mathcal{A}_{l-1}\big\},\
l=2,\cdots,j_0.
\end{eqnarray*}
Assume  $a_1=(a_1^1,\cdots,a_1^m)^{*},~a_2=(a_2^1,\cdots,a_2^n)^{*},$ $a=(a_1^*,a_2^*)^*.$
 $\mathbb{N}=\{1,\cdots\}$.
Let  $\det(A)$ be the determinant of the matrix $A=(a_{i,j})$,
$\|A\|^2=\sum_{i,j}a^2_{i,j}$. Let $\langle\cdot,\cdot\rangle$  be
the Euclidean inner product and $|\cdot|$ be  the Euclidean norm.
For any $x_0\in\mathbb{R}^{m+n}$ and $R>0$, $\overline{B}(x_0,R)=\{x\in \mathbb{R}^{m+n},|x-x_0|\leq R\}$,
 $B^{\circ}(x_0,R)=\{x\in \mathbb{R}^{m+n},|x-x_0|< R\}.$
$\|k\|_{\infty}$ denotes the essential supreme norm for the
function $k$ defined on Euclidean space. We use $C(d)$ or
$\epsilon_0(d)$ to  denote a positive and finite constant depending
on $d, \ \|\nabla a\|_{\infty}$ and $ \|\nabla b\|_{\infty}$.
This constant   may change from line to line. Sometimes, we will use
$C$  instead of $C(d)$ for the convenience of writing. Without otherwise
specified, in this article, $(x_t,y_t)$ is the solution for
Eq.$(\ref{1-1})$ and  $(x,y)$ is its   initial value.
Let $M_t$ be  the Malliavin matrix for $(x_t,y_t)$.
Then  (c.f. \cite{Nualart})
\begin{align}\label{2-2}
M_t=J_{t}\int_{0}^{t}J_{s}^{-1}\Big(
\begin{matrix}
 0 \\
 b(x_{s},y_{s})%
\end{matrix}%
\Big) \Big(
\begin{matrix}
0 \\
 b(x_{s},y_{s})%
\end{matrix}%
\Big) ^{*}(J_{s}^{-1})^{*}ds.J_{t}^{* },
\end{align}
here $J_t^{-1}$ satisfies
\begin{eqnarray}
\label{17}
\begin{split}
J_{t}^{-1}&=I_{m+n} -\int_0^t J_{s}^{-1}\Big(
\begin{matrix}
0 & 0 \\
\nabla_x b_{j}  & \nabla _{y}b_{j}%
\end{matrix}%
\Big)(x_s,y_s) dW_{j}(s)
\\   & \ \ \ -\int_0^t J_{s}^{-1}\Bigg[ \Big(
\begin{array}{cc}
\nabla _{x}a_{1} & \nabla _{y}a_{1} \\
\nabla _{x}a_{2} & \nabla _{y}a_{2}%
\end{array}%
\Big)(x_s,y_s)
\\ &\ \ \ \ \ \ \ \ \ \ \ \ \ \ \ \ \ -\sum_{j=1}^d\Big(
\begin{array}{cc}
0 & 0 \\
\nabla _{y}b_{j}\nabla_x b_{j} & \nabla_y b_j \nabla_y b_j
\end{array}%
\Big)(x_s,y_s)  \Bigg] ds,
\end{split}
\end{eqnarray}
and $J_t$ satisfies
\begin{equation}\label{2-1-1000}
\begin{split}
J_t&=I_{m+n}+\int_0^t \Big(\begin{matrix}
\nabla _{x}a_{1} & \nabla _{y}a_{1} \\
\nabla _{x}a_{2} & \nabla _{y}a_{2}%
\end{matrix}\Big)(x_s,y_s)J_sds\\
&  \ \ \ +\sum_{j=1}^d \int_0^t \Big(
\begin{matrix}
0 & 0 \\
\nabla _{x}b_{j} & \nabla _{y}b_{j}%
\end{matrix}%
\Big)(x_s,y_s) J_sdW_j(s).
\end{split}
\end{equation}

Our article is organized as follows.
In section $2$, we prove the key theorem of this article  Theorem $\ref{15}$ under the  Hypothesis $\ref{2-18}$.
 In Hypothesis $\ref{2-18}$, we  only need $a_2\in C^1,~b\in
C^2$~and $a_1\in C^{j_0+2}$ for some $j_0\in \mathbb{N}$.
Compare with H\"{o}rmander condition, the ¡§
functions $a_2$ and $b$ are only required to be $C^1$
and $C^2$ respectively.
 Our method to prove Theorem $\ref{15}$  is similar to that
 in \cite{Nualart},  but it also has some differences.   These differences  depend heavily on the special form of
the Eq.$(\ref{1-1})$.
In \cite{Nualart}, $ J_t^{-1}$  is regarded as  a whole.
Here, we divide   $ J_t^{-1}$
 into   $\left(
\begin{matrix}
A_{t} & B_{t} \\
C_{t} & D_{t}
\end{matrix}
\right)$ and do more elaborate estimates.

In section $3$, we firstly give a  local uniform estimate for Malliavin matrix under the Hypothesis $\ref{100-1}$,  and then give a gradient estimate in  Theorem $\ref{2-39}$.
The local uniform estimate  for Malliavin matrix    is a key point to prove   Theorem $\ref{2-39}$.

In section 3, we have  proved  $P_t$ is  strong Feller under some conditions which need
all the coefficients of  Eq.$(\ref{1-1})$ are in  $C_b^2$.
Since there are     bounded conditions  on the  coefficients and their  derivatives, it   seems too strong to apply, for example,
the Hamiltonian  systems, so we weaken this bounded conditions in  section 4. In the  section 4,  we mainly use the  localization method to prove $P_t$ is strong Feller, our hypothesis  is the Hypothesis  $\ref{3-20}$.

In section $5$, we apply  the above results to some examples, such
as the Lagevin SDEs, the stochastic Hamiltonian systems and high
order stochastic differential equations.

\section{The $L^p$ Integrability of  the Inverse of Malliavin Matrix}
In subsection 2.1, we give the key Theorem  $\ref{15}$  and put it's proof in subsection 2.2.
\subsection{The Main Theorem and  Its Relations with H\"{o}rmander Theorem}

\begin{hypothesis}
\label{2-18}    $(x,y)\in \mathbb{R}^m\times \mathbb{R}^n$ and there  exists
 a $j_0:=j_0(x,y) \in \mathbb{N}$ such that:
\begin{description}
  \item[(\romannumeral1)]$a_1 \in C_b^1(\mathbb{R}^m\times \mathbb{R}^n;\mathbb{R}^m) \cap C^{j_0+2}(\mathbb{R}^m\times
\mathbb{R}^n;\mathbb{R}^m),
a_2\in C_b^1(\mathbb{R}^m\times \mathbb{R}^n;\mathbb{R}^n);$
\item[(\romannumeral2)] $\det(b(x,y)\cdot b^{*}(x,y)) \neq 0$,  $b \in C_b^{1}(\mathbb{R}^m \times \mathbb{R}^n;\mathbb{R}^{n}\times \mathbb{R}^d)
\cap C^2(\mathbb{R}^m \times \mathbb{R}^n;\mathbb{R}^{n}\times
\mathbb{R}^d);$
 \item[(\romannumeral3)] The vector space spanned by  $\cup_{k=1}^{j_0} \mathcal{A}_{k}$ at point $(x,y)$ has dimension $m$.
\end{description}
\end{hypothesis}
\begin{thm}
\label{15}  Let the     Hypothesis $\ref{2-18}$  hold, $T>0$,   then $\det(M_{T}^{-1})\in L^{p}(\Omega,\mathbb{P}_{x,y})$ for any $p>0$.
\end{thm}
\begin{Rem}
If the   coefficients   $a_1,a_2,b$ in Eq.$(\ref{1-1})$ also depend on $t$ and for any $T>0$, $t\rightarrow (a_1(t,0),a_2(t,0))$ and $t\rightarrow b(t,0)$ are bounded on $[0,T]$,  then the  Theorem $\ref{15}$, Theorem
 $\ref{2-39}$ and   Theorem $\ref{4-16}$   hold  also.
\end{Rem}

There is a natural relation between H\"{o}rmander conditions $(\mathbf{H})$
and Hypothesis $\ref{2-18}$
from the well-known geometric interpretation of H\"{o}rmander conditions. This relation   can  be proved directly by tedious calculations also. 
\begin{Rem}{\label{2-11}}
Assume $a_{1}\in C^{\infty }(\mathbb{R}^{m}\times
\mathbb{R}^{n};\mathbb{R}^{m})$, $a_{2}\in C^{\infty
}(\mathbb{R}^{m}\times \mathbb{R}^{n};\mathbb{R}^{n}),$ $b\in C^{\infty
}(\mathbb{R}^{m}\times \mathbb{R}^{n};\mathbb{R}^n\times
\mathbb{R}^d)$,  $n\leq d$,  $\det(b(x,y)\cdot b^{*}(x,y))\neq 0$.
Then the  H\"{o}rmander conditions  $(\mathbf{H})$ is equivalent to Hypothesis $\ref{2-18}$.
\end{Rem}

But Hypothesis $\ref{2-18}$ is weaker than  H\"{o}rmander conditions in some sense, the followings   are three examples.
\begin{example}
A  concrete example is the following stochastic differential
equation
\begin{eqnarray*}
\begin{cases}
dx_1(t)=x_2(t)dt+y_tdt\\
dx_2(t)=x_1(t)dt\\
dx_3(t)=x_2(t)dt+x_3(t)dt\\
dy_t=a_2(x_t,y_t)dt+b dW_t
\end{cases},
\end{eqnarray*}
where $x_t=(x_1(t),x_2(t),x_3(t))^{*}\in \mathbb{R}^3$, $y_t\in
\mathbb{R}^1$,  $a_2(x_1,x_2,x_3,y)$  only has one order
derivatives    and $b \in \mathbb{R}^1 \setminus\{0\}$
is a constant, then the  Hypothesis $\ref{2-18}$ holds, but the
H\"{o}rmander conditions   $(\mathbf{H})$  can't  be applied directly.
\end{example}
\begin{proof}
Set $a_1(x_1,x_2,x_3,y)=(x_2+y,x_1,x_2+x_3)^{*}$,  then
\begin{eqnarray*}
\nabla _x a_1=\left(\begin{matrix} 0&1& 0\\
1 &0&0\\
0&1&1
\end{matrix}\right),~\nabla _y a_1=\left(\begin{matrix} 1\\
0\\
0
\end{matrix}\right)
\end{eqnarray*}
 In this example, by calculating,
 \begin{eqnarray*}
 \mathcal{A}_1=\nabla _y a_1,~ \mathcal{A}_2= -\nabla _x a_1 \nabla _y a_1,~ \mathcal{A}_3= +(\nabla _x a_1)^2 \nabla _y a_1
 \end{eqnarray*}
and
 \begin{eqnarray*}
 \nabla _y a_1=\left(\begin{matrix} 1\\
0\\
0
\end{matrix}\right),    ~-\nabla _x a_1 \nabla _y a_1=-\left( \begin{matrix} 0\\1\\0\end{matrix}   \right),~ +(\nabla _x a_1)^2 \nabla _y a_1=\left(\begin{matrix} 1\\0\\1\end{matrix}\right)
 \end{eqnarray*}
 So  the vector space spanned  by~$ \{\mathcal{A}_{j},~j=1,2,3\}$ at any
point~$(x,y)$ is  $\mathbb{R}^3.$
\end{proof}

The following example is a special case for the SDE considered  in   \cite{Delarue}  with $n=3.$
\begin{example}
Consider the following SDE
\begin{eqnarray*}
\left\{
  \begin{split}
X_t^1&=x_1+\int_0^t F_1(s,X_s^1,X_s^2,X_s^3)ds+\int_0^t\sigma(s,X_s^1,X_s^2,X_s^3)dW_s,
\\
X_t^2&=x_2+\int_0^t F_2(s,X_s^1,X_s^2,X_s^3)ds,
\\
X_t^3&=x_3+\int_0^t F_3(s,X_s^2,X_s^3)ds.
\end{split}
\right.
\end{eqnarray*}
  If $\det(\sigma(0,x_1,x_2,x_3)\sigma^{*}(0,x_1,x_2,x_3))\neq 0$, by calculating,
\begin{eqnarray*}
  \mathcal{A}_1=\Bigg\{\Big(\begin{matrix} \nabla_{x_1}F_2\\
0
\end{matrix}\Big)\Bigg\}, \ \  \mathcal{A}_2=\Bigg\{\Big(\begin{matrix} \nabla_{x_1x_1}F_2\\
0
\end{matrix}\Big), \Big(\begin{matrix} G(x_1,x_2,x_3)\\
\nabla_{x_2}F_3\cdot \nabla_{x_1}F_2
\end{matrix}\Big) \Bigg\},
\end{eqnarray*}
for some function $G$.
The condition  in \cite{Delarue} is  $\nabla_{x_1}F_2 \cdot  \nabla_{x_2}F_3\neq 0.$
So the $\bf{(iii)}$  in Hypothesis $\ref{2-18}$  is the same as  that  in \cite{Delarue}.
And the
H\"{o}rmander conditions   $(\mathbf{H})$  can't  be applied directly.
\end{example}

The following example shows the condition $\bf{(ii)}$ in Hypothesis $\ref{2-18}$ is necessary in some sense.
\begin{example}
$W_t$ is a one dimension standard Brownian motion  and
$X_t=(X_t^1,X_t^2),Y_t=(Y_t^1,Y_t^2)$ satisfy the following equations
 \begin{eqnarray*}
\left\{
  \begin{split}
  X_t&= \Big(\begin{matrix}
  1 &0
\\ 0 &1
\end{matrix}\Big)Y_tdt,
\\ Y_t&=\Big(\begin{matrix}
  1 &1
\\ 0 &1
\end{matrix}\Big)Y_t
dt+\Big(\begin{matrix}
  1\\
1
\end{matrix}\Big)
dW_t.
\end{split}
\right.
\end{eqnarray*}
If set
\begin{eqnarray*}
V=\left(\begin{matrix}
  0 & 0 &1 &1 \\
  2& 2 &1 &-1\\
  2 & -3 &1 &-1 \\
  1 & 0 &-1 &1
\end{matrix}\right) \text{and} \
\left(
\begin{matrix}
  \overline{X}^1_t
\\
  \overline{X}^2_t
\\
  \overline{Y}^1_t
\\
 \overline{Y}^2_t
\end{matrix}\right)=V\left(
\begin{matrix}
  X^1_t
\\
  X^2_t
\\
 Y^1_t
\\
Y^2_t
\end{matrix}\right),
\end{eqnarray*}
then  $\overline{Y}_t^2\equiv0$,  so the Malliavin matrix for
$( \overline{X}^1_t,
  \overline{X}^2_t,
  \overline{Y}^1_t,
 \overline{Y}^2_t)$ singular a.s..
For  $V$ is invertible, so the the Malliavin matrix for $(X_t,Y_t)$ is  singular also.
 But   the Malliavin matrix for $Y_t$ is invertible by H\"{o}rmander Theorem,   and
the $\bf{(i),(iii)}$  in the  Hypothesis $\ref{2-18}$  hold.
 \end{example}

\subsection{Proof of Theorem  $\ref{15}$}
In \cite{Nualart}, the inverse of Jacobian matrix $ J_t^{-1}$ is regarded as a whole.   In this subsection,  we divide  $J_t^{-1}$ into four parts $\left(\begin{matrix}A_{t} & B_{t} \\C_{t} & D_{t}\end{matrix}%
\right)$  and  do more elaborate estimates,  we obtain  $%
\det(M_{T}^{-1}) \in L^{p}(\Omega,\mathbb{P}_{x,y}),$ $\forall p,T>0$  under some conditions weaker than H\"{o}rmander conditions.  Our  main method is similar  to that
 in \cite{Nualart},  but it also has some differences and its proof  is more complicated.
The  differences  depend heavily on the special form of
the Eq.$(\ref{1-1}).$
Before we   prove the  Theorem  $\ref{15}$, we introduce some notations and list the  Lemmas which will be used in the proof of Theorem   $\ref{15}$.

Assume $J_t^{-1}=\left(
\begin{matrix}
A_{t} & B_{t} \\
C_{t} & D_{t}
\end{matrix}
\right)$, $A_t$ is a matrix with dimension $m\times m,$  then
\begin{eqnarray}\label{2-5}
\left\{
\begin{split}
&dA_t=-\sum_{j=1}^dB_t \nabla_x b_j dW_j(t)-(A_t\nabla_x a_1+B_t \nabla_x a_2)dt+\sum_{j=1}^dB_t \nabla_y b_j  \nabla_x b_j dt,
\\&dB_t=-\sum_{j=1}^dB_t \nabla_y b_j dW_j(t)-(A_t\nabla_y a_1+B_t \nabla_y a_2)dt+\sum_{j=1}^dB_t \nabla_y b_j  \nabla_y b_j dt,
\\&dC_t=-\sum_{j=1}^dD_t \nabla_x b_j dW_j(t)-(C_t\nabla_x a_1+D_t \nabla_x a_2)dt+\sum_{j=1}^dD_t \nabla_y b_j  \nabla_x b_j dt,
\\&dD_t=-\sum_{j=1}^dD_t \nabla_y b_j dW_j(t)-(C_t\nabla_y a_1+D_t \nabla_y a_2)dt+\sum_{j=1}^dD_t \nabla_y b_j  \nabla_y b_j dt.
\end{split}\right.
\end{eqnarray}
For   the vector space spanned by  $\cup_{k=1}^{j_0} \mathcal{A}_{k}$ at point $(x,y)$ has dimension $m,$ then there exist two positive constants $R_1\ $and
$c$ such that
\begin{eqnarray} \label{2-2-1}
\sum_{j=1}^{j_{0}}\sum_{V\in \mathcal{A}_{j}}(v^{* }V(x',y'))^{2}\geq c
\end{eqnarray}
holds for all $v \in \mathbb{R}^m, ~|v|=1$ and $|(x',y')-(x,y)|\leq R_1$.

Fix  $R_2=\frac{1}{100}$, define the stopping time as
\begin{eqnarray}\label{2-122}
S=S(x,y):=\inf\big\{s\geq 0: \sup_{0\leq u\leq s}|(x_{u},y_{u})-(x,y)|\geq R_1 \ \text{or}\sup_{0\leq u\leq s}|J_{u}^{-1}-I_{m+n}|\geq R_2 \big\}.
\end{eqnarray}
Define the  adapted process
\begin{eqnarray}\label{2-12}
\lambda (s)=\inf_{|v|=1}\{v^{*
}b(x_{s},y_{s})b^{* }(x_{s},y_{s})v\}.
\end{eqnarray}
For  $ |\inf_{v}a_{v}-\inf_{v}b_{v}|\leq
\sup_{v}|a_{v}-b_{v}|$, so
\begin{eqnarray}\label{2-6}
\left\vert \lambda (s)-\lambda (t)\right\vert
\leq \| b(x_{s},y_{s})b^{*
}(x_{s},y_{s})-b(x_{t},y_{t})b^{* }(x_{t},y_{t})\|.
\end{eqnarray}
Then $\lambda(s)$ is continuous w.r.t $s$.
Since det$(b(x,y) b^{*}(x,y)) \neq 0,$ so $\lambda (0)>
0$.   For $R_3=\lambda(0)/2$,  we  define the stopping times
\begin{eqnarray} \label{2-137}
&&\tau'=\inf\{s>0:|\lambda (s)-\lambda (0)|\geq R_3
\},\\
&&\tau =\tau'\wedge S \wedge T.
\end{eqnarray}
Let  $j_0$  be as  in  Hypothesis $\ref{2-18}$. $~v=(v_{1}^*,v_{2}^*)^{*} \in \mathbb{R}^{m}\times \mathbb{R}^{n}$  with  $|v|=1$. Fix $q>8$  and set
\begin{eqnarray*}
&&F=\Bigg\{\sum_{j=1}^{j=d}\int_{0}^{T}|(v_{1}^{*}B_{s}+v_{2}^{*
}D_{s})b_{j}|^{2}ds\leq \epsilon ^{q^{3j_{0}+6}}\Bigg\},\\
&&E_{j}=\Bigg\{\sum_{K\in \mathcal{A}_{j}}\int_{0}^{\tau
}|(v_{1}^{*}A_{s}+v_{2}^{*}C_{s})K(x_{s},y_{s})|^{2}ds\leq
\epsilon ^{q^{3j_{0}+3-3j}}\Bigg\}, ~j=1,\cdots,j_0,\\
&&E=F\cap E_{1}\cap E_{2}\cdots\cap E_{j_{0}}.
\end{eqnarray*}
Denote by
\begin{eqnarray}\label{3-2-1}
\|v_{1}^{*}B_{.}+v_{2}^{*}D_{.}\|^{2}_{\frac{1}{4}}:=\sup_{s,r\in
\lbrack 0,\tau ]}\frac{\Big|\big(|v_{1}^{*}B_{s}+v_{2}^{*
}D_{s}|^{2}-|v_{1}^{*}B_{r}+v_{2}^{*}D_{r}|^{2}\big)\Big|}{
|s-r|^{\frac{1}{4}}}.
\end{eqnarray}
\begin{Rem}
In the definition of  $S, \ R_2=\frac{1}{100}$ is a technique  skill.
In the Lemma $\ref{2-1-6},$ we essentially  need $R_2$ small enough,  and  we  need $R_2$ is  finite in  other places.  Here, $R_1, R_3$ and $c$  depend on $(x,y)$.
\end{Rem}
From  the definition of $S$,  when $s\leq S$, $|(x_s,y_s)-(x,y)|\leq R_1$. So from $(\ref{2-2-1})$,
\begin{align} \label{2-2-3}
\sum_{j=1}^{j_{0}}\sum_{V\in \mathcal{A}_{j}}(v^{* }V(x_s,y_s))^2\geq c
\end{align}
holds for all $ s \leq S$ and  $v \in \mathbb{R}^m$ with  $|v|=1$.

\begin{lemma}
\label{2-22} (c.f. Lemma $6.14$, \cite{Hairer}). Let $f:[0,T_0]\rightarrow \mathbb{R}$ be continuous
differentiable and $ \alpha \in (0,1]$. Then
\begin{equation*}
\|\partial_t f\|_{\infty}=\|f\| _{1}\leq 4 \parallel f\parallel _{\infty}\max\big\{\frac{1}{T_0},~\parallel f\parallel _{\infty }^{-\frac{1}{1+\alpha }%
}\parallel \partial _{t}f\parallel _{\alpha }^{\frac{1}{1+\alpha
}}\big\},
\end{equation*}
where $\|f\|_{\alpha }=\sup\limits_{s,t \in [0,T_0],s \neq  t}\frac{|f(t)-f(s)|}{%
|t-s|^{\alpha }}$.
\end{lemma}

\begin{lemma}
\label{2-15} (c.f. Corollary 2.2.1, \cite{Nualart}). Let the  Hypothesis $\ref{2-18}$ hold,  then for any $p,T>0$,  there exists a finite constant $ C(T,p,x,y)$  such that
\begin{eqnarray*}
 \mathbb{E}\Big\{\sup_{0\leq t\leq T}|(x_t,y_t)|^{p}\Big\}\leq C(T,p,x,y).
\end{eqnarray*}
\end{lemma}

\begin{lemma}
\label{32} Let the  Hypothesis $\ref{2-18}$ hold, then  for any $p,T>0,$   there exists a finite constant $ C(T,p,x,y)$  such that
\begin{eqnarray*}
&&\mathbb{E}\Big\{\sup_{0\leq s\leq T}\|J_s^{-1}\|^{p}\Big\}\leq C(T,p,x,y),
\\ \nonumber
&&\mathbb{E}\Big\{\sup_{0\leq s\leq T}\|J_s\|^{p}\Big\}\leq C(T,p,x,y).
\end{eqnarray*}

\end{lemma}

\begin{proof}
It directly   follows from   $(\ref{17})$  $(\ref{2-1-1000})$  and    Lemma 2.2.1   in  \cite{Nualart}.
\end{proof}
\begin{lemma}
\label{22}
Let the  Hypothesis $\ref{2-18}$ hold, then  for any $p>0$,    there exists a finite constant $ C(p,x,y)$  such that
\begin{eqnarray*}
\mathbb{P}\{S<\epsilon \}\leq  C(p,x,y) \epsilon ^{p}, ~ \forall \epsilon>0.
\end{eqnarray*}
\end{lemma}
\begin{proof} For any $p>0$, it holds that
\begin{eqnarray*}
\begin{split}
 \mathbb{P}\Big\{S<\epsilon \Big\}&\leq \mathbb{P}\Big\{\sup_{0\leq s\leq \epsilon }|(x_{s},y_{s})-(x,y)|\geq R_1\Big\}
 +\mathbb{P}\Big\{\sup_{0\leq s\leq \epsilon }|J_{s}^{-1}-I_{m+n}|\geq R_2\Big\} \\
&\leq C(p,x,y)\mathbb{E} \Big\{ \sup_{0\leq s\leq \epsilon }|(x_{s},y_{s})-(x,y)|^{2p}\Big\}
 +C(p,x,y)\mathbb{E}\Big\{ \sup_{0\leq s\leq
\epsilon }|J_{s}^{-1}-I_{m+n}|^{2p}\Big\}.
\end{split}
\end{eqnarray*}
Then this Lemma comes from Burkholder's and H\"{o}lder's inequalities.
\end{proof}
\begin{lemma}
\label{2-290}
Let the  Hypothesis $\ref{2-18}$ hold, then  for any $p>0$,    there exists a finite constant $ C(p,T,x,y)$  such that
\begin{eqnarray*}
\mathbb{P}\{\tau<\epsilon \}\leq  C(p,T,x,y) \epsilon ^{p}, ~ \forall \epsilon>0.
\end{eqnarray*}
\end{lemma}
\begin{proof}
From Lemma  $\ref{22}$ and the fact
\begin{eqnarray*}
\mathbb{P}\{\tau <\epsilon ^{ }\}
&\leq &\mathbb{P}\{S<\epsilon ^{}\}+\mathbb{P}\{\tau' <\epsilon ^{ }\}+\mathbb{P}\{T <\epsilon \},
\end{eqnarray*}
 we only need to estimate  $\mathbb{P}\{\tau' <\epsilon ^{ }\}.$
For any $p>0$
\begin{eqnarray}\label{2-4}
\begin{split}
\mathbb{P}\{\tau'<\epsilon^{ }\} &\leq  \mathbb{P}\Big\{\sup_{0\leq
r\leq \epsilon ^{ }}|\lambda (s)-\lambda (0)|\geq R_3 \Big\}
\\
&\leq C(p,x,y)\mathbb{E}\Big\{\sup_{0\leq s\leq \epsilon ^{ }}|\lambda
(s)-\lambda (0)|^{2p}\Big\}.
\end{split}
\end{eqnarray}
Due to    $ |\inf_{v}a_{v}-\inf_{v}b_{v}|\leq \sup_{v}|a_{v}-b_{v}|,$
\begin{equation}\label{3-9}
  \begin{split}
    &\mathbb{E} \Big\{ \sup_{0\leq s\leq \epsilon ^{ }}|\lambda (s)-\lambda(0)|^{2p}\Big\}\\
    &\leq C(p)\sum_{\substack{ i,k=1,\cdots ,n \\ j=1,\cdots ,d}} \mathbb{E}\Big\{\sup_{0\leq s\leq
\epsilon ^{ }}\Big| b_{kj}(x_{s},y_{s})b_{ij}(x_{s},y_{s}) -b_{kj}(x,y)b_{ij}(x,y)\Big|^{2p}\Big\}.
  \end{split}
\end{equation}
Note that
\begin{eqnarray}\label{3-10}
\begin{split}
&\ b_{kj}(x_{s},y_{s})b_{ij}(x_{s},y_{s})-b_{kj}(x,y)b_{ij}(x,y)
\\  &=(b_{kj}(x_{s},y_{s})-b_{kj}(x,y))(b_{ij}(x_s,y_s)-b_{ij}(x,y))
\\  &\ \ \  +b_{kj}(x,y)(b_{ij}(x_s,y_s)-b_{ij}(x,y)
\\ & \ \ \  +b_{ij}(x,y)(b_{kj}(x_{s},y_{s})-b_{kj}(x,y)).
\end{split}
\end{eqnarray}
From  $(\ref{2-4})(\ref{3-9})(\ref{3-10})$,
\begin{eqnarray*}
\mathbb{P}\{\tau'<\epsilon \}\leq  C(p,x,y) \Bigg[\mathbb{E} \Big\{ \sup_{0\leq s\leq \epsilon ^{ }}\left|
b_{ij}(x_{s},y_{s})-b_{ij}(x,y)\right| ^{2p}\Big\}+\mathbb{E}\Big\{ \sup_{0\leq s\leq \epsilon ^{ }}\left|
b_{ij}(x_{s},y_{s})-b_{ij}(x,y)\right| ^{4p}\Big\}\Bigg].
\end{eqnarray*}
So this Lemma comes from Burkholder's and H\"{o}lder's inequalities  and the fact
\begin{eqnarray*}
b_{ij}(x_{s},y_{s})-b_{ij}(x,y)=\langle \nabla b_{ij}(\xi,\eta),(x_s,y_s)-(x,y)\rangle,
\end{eqnarray*}
here  $(\xi,\eta)$ is some point
depending  on $(x_s,y_s)$ and $(x,y).$
\end{proof}

\begin{lemma}\label{23}
Let~$\sigma$~be a finite stopping time with bound $c_{\sigma} <\infty$,
and  there exists~$\tilde{p}>0$~such that
\begin{eqnarray*}
\label{2-33}
\mathbb{P}\{\sigma<\epsilon ^{ }\}\leq C(c_{\sigma},\tilde{p})\epsilon ^{
\tilde{p}}, ~ \forall  \epsilon >0,
\end{eqnarray*}
holds for some constant $C(c_\sigma,\tilde{p}).$  Assume  $\gamma (t)=(\gamma _{1}(t),...,\gamma _{d}(t)), ~u(t)=(u_{1}(t),...u_{d}(t))$ are continuous  adapted processes,  $W(t)=(W_1(t),\cdots,W_d(t))^{*}$ is a  standard Wiener process,  $a(t), \tilde{y}(t)\in \mathbb{R}$~and for $t \in [0,c_\sigma]$
\begin{eqnarray*}
\label{2-200}
a(t) &=&\alpha +\int_{0}^{t}\beta (s)ds+\int_{0}^{t}\gamma
(s)dW(s), \\
\label{2-201}
\tilde{y}(t) &=&\tilde{y}+\int_{0}^{t}a(s)ds+\int_{0}^{t}u(s)dW(s).\
\end{eqnarray*}%
Suppose for some $p,\tilde{c}>0,$
\begin{eqnarray*}
\label{2-126}
\mathbb{E}\Big\{ \sup_{0\leq t\leq \sigma }\left(|\beta (t)|+|\gamma
(t)|+|a(t)|+|u(t)|\right)^{p}\Big\}\leq \tilde{c}<\infty .
\end{eqnarray*}
Then for any  three positive  numbers~$(q,r,v)$ satisfy
  $2q-36r-9v>16$,  there exists~$\epsilon
_{0}=\epsilon _{0}(c_\sigma,q,r,v)$~such that for any $\epsilon
<\epsilon _{0},$
\begin{eqnarray*}
\mathbb{P}\left( \int_{0}^{\sigma}\tilde{y}(t)^{2}dt<\epsilon ^{q},\int_{0}^{\sigma
}(|a(t)|^{2}+|u(t)|^{2})dt\geq \epsilon \right) \leq \tilde{c}\epsilon
^{rp}+\exp{(-\epsilon ^{-\frac{v}{4}})}+C(c_\sigma,\tilde{p})\epsilon ^{\tilde{p}}.
\end{eqnarray*}
\end{lemma}
The proof of Lemma $\ref{23}$ is postponed to  Appendix $\ref{A-2}$

\begin{lemma}\label{25}
Let $\sigma $ be a finite stopping time with bound $c_\sigma<\infty$,
and  there exists  $\tilde{p}>2$, such that
\begin{eqnarray} \label{2-100}
\mathbb{P}\{\sigma <\epsilon ^{ }\}\leq  C(c_\sigma,\tilde{p}) \epsilon ^{
\tilde{p}},~\forall \epsilon >0
\end{eqnarray}
holds for some constant $C(c_\sigma,\tilde{p}).$  Consider the following one dimensional stochastic differential equation
\begin{eqnarray}
\label{2-204}
\tilde{y}(t)=\tilde{y}+\int_{0}^{t}a(s)ds+\int_{0}^{t}u(s)dW(s),~t\in \lbrack 0,c_\sigma],
\end{eqnarray}
where $u(s)=(u_1(s),\cdots,u_d(s))$ is a  continuous adapted  process,  $W(t)=(W_1(t),\cdots,W_d(t))^{*}$ is a d-dimensional standard
Wiener process.  $a(t),u(t)$  satisfy
\begin{eqnarray} \label{2-101}
\mathbb{E}\Big\{  \sup_{0\leq t\leq \sigma }(|a(t)|+|u(t)|)^{p}\Big\}  \leq  \tilde{c}<\infty,
\end{eqnarray}%
for some $p,\tilde{c}>0$.

Then for any  three  positive numbers~$(q,r,v)$~satisfy
$2q>8+20r+v$,  there exists $\epsilon
_{0}=\epsilon _{0}(c_\sigma,q,r,v)$~such that for any $\epsilon \leq \epsilon _{0}$,
\begin{eqnarray*}
\mathbb{P}\left(\int_{0}^{\sigma }\tilde{y}(t)^{2}dt<\epsilon ^{q},~\int_{0}^{\sigma}|u(t)|^{2}dt\geq
\epsilon \right)\leq \tilde{c}\epsilon ^{rp}+\exp \{-\epsilon ^{-\frac{v}{4}}\}+ C(c_\sigma,\tilde{p})
\epsilon ^{ \tilde{p}}.
\end{eqnarray*}
\end{lemma}
The proof of Lemma $\ref{25}$ is postponed to  Appendix $\ref{A-2}$

\begin{lemma}
\label{29}  Let the  Hypothesis $\ref{2-18}$ hold and $C_0=2/\lambda(0)$, then for any  $p>0$,  there  exists a  constant $C=C(p,T,x,y,q)$
 such that
\begin{eqnarray*}
\mathbb{P}\Big\{\|v_{1}^{*}B_{.}+v_{2}^{* }D_{.}\|^{2}_{\frac{1}{4}}>\frac{1}{4^{\frac{5}{4}}C_0^{\frac{1}{4}}}\epsilon
^{-\frac{q^{3j_{0}+6}}{8}}\Big\}\leq C(p,T,x,y,q) \epsilon ^{p},~\forall \epsilon>0.
\end{eqnarray*}
\end{lemma}

\begin{proof}
From  $(\ref{2-5})$ and It\^{o}'s formula
\begin{align*}
 d|v_{1}^{*}B_{s}+v_{2}^{*}D_{s}|^{2}&= -2\langle (v_{1}^{*
}B_{s}+v_{2}^{*}D_{s}),(v_{1}^{*}B_{s}+v_{2}^{*}D_{s})
\nabla_y a_{2} \rangle ds
\\ &\ \ \  -2\langle (v_{1}^{*}B_{s}+v_{2}^{*}D_{s}),(v_{1}^{*
}A_{s}+v_{2}^{*}C_{s})\nabla _{y}a_{1}\rangle ds \\ \nonumber
&\ \ \  -2\sum_{j=1}^d\langle (v_{1}^{*}B_{s}+v_{2}^{*}D_{s}),(v_{1}^{*
}B_{s}+v_{2}^{*}D_{s})\nabla _{y}b_{j}\rangle dW_{j}(s)
\\&\ \ \  +2\sum_{j=1}^d\langle (v_{1}^{*}B_{s}+v_{2}^{*}D_{s}),(v_{1}^{*
}B_{s}+v_{2}^{*}D_{s})\nabla _{y}b_{j}\nabla _{y}b_{j}\rangle ds \\  \label{1-10-1}
&\ \ \  +\sum_{j=1}^d\langle (v_{1}^{*}B_{s}+v_{2}^{*}D_{s})\nabla
_{y}b_{j},(v_{1}^{*}B_{s}+v_{2}^{*}D_{s})\nabla _{y}b_{j}\rangle ds.
\end{align*}
From BDG inequality, Lemma $\ref{32}$ and and the above equation,
for any $p>0$, there exists a  constant $C=C(p,T,x,y)$ such that for
any $s,r \in [0,T],$
\begin{eqnarray*}
&&\mathbb{E} \left[ |v_{1}^{*}B_{s}+v_{2}^{*
}D_{s}|^{2}-|v_{1}^{*}B_{r}+v_{2}^{*}D_{r}|^{2}\right]
^{2p} \leq C|s-r|^{p}.
\end{eqnarray*}
Set   $\gamma =2p,~\epsilon =p-1,~T_0=T$  in Theorem 2.1, \cite{REVUZ},   for any $p>2,$
\begin{eqnarray*}
C_{p,T,x,y}:=\mathbb{E}\left[ \|v_{1}^{*}B_{.}+v_{2}^{*}D_{.}\|^{2}_{\frac{1}{4}}%
\right] ^{2p}<\infty .
\end{eqnarray*}
Thus, $\forall \epsilon>0,~\forall p'>0$
\begin{eqnarray}
\label{2-124}
\begin{split}
&\mathbb{P}\Big\{\|v_{1}^{*}B_{.}+v_{2}^{* }D_{.}\|^{2}_{\frac{1}{4}}>\frac{1}{4^{\frac{5}{4}}C_0^{\frac{1}{4}}}\epsilon
^{-\frac{q^{3j_{0}+6}}{8}}\Big\}
\\  & \leq C(p') \epsilon^{\frac{q^{3j_{0}+6}}{8}p'}\mathbb{E}\left[ \|v_{1}^{*}B_{.}+v_{2}^{*}D_{.}\|^{2}_{\frac{1}{4}}%
\right] ^{p'}.
\end{split}
\end{eqnarray}
Then this Lemma is   obtained by setting   $p'=\frac{8p}{q^{3j_0+6}}$  in $(\ref{2-124})$,
\end{proof}

\begin{lemma}
\label{19}
Let the  Hypothesis $\ref{2-18}$ hold,
then for  any $p>0$   there exists a constant
$C(p,T,x,y,q)$  such that
\begin{eqnarray*}
\mathbb{P}\Big\{F\cap \big\{\sup_{s\in \lbrack 0,\tau ]}|v_{1}^{*
}B_{s}+v_{2}^{*}D_{s}|^{2}>\epsilon ^{\frac{q^{3j_{0}+6}}{10}}\big\}\Big\}\leq
C(p,T,x,y,q)\epsilon ^{p},~\forall \epsilon >0.
\end{eqnarray*}

\end{lemma}

\begin{proof}
Due to $\tau\leq\tau'
$ and $\omega \in F,$  there exists a constant
$C_0=2/\lambda(0)=C(x,y)>0$ such that
\begin{eqnarray*}
\int_{0}^{\tau }|v_{1}^{*}B_{s}+v_{2}^{*}D_{s}|^{2}(\omega)ds\leq
C_0 \epsilon ^{q^{3j_{0}+6}}.
\end{eqnarray*}
 Set $%
f(s)=\int_{0}^{s}|v_{1}^{*}B_{u}+v_{2}^{*}D_{u}|^{2}du,~T_0=\tau(\omega)$  and   $\alpha =\frac{1}{4}$  in  Lemma $\ref{2-22}$, then
\begin{eqnarray*}
&&\sup_{s\in \lbrack 0,\tau ]}|v_{1}^{*}B_{s}+v_{2}^{*
}D_{s}|^{2}\leq \\
&&\begin{split}
&\max\Bigg\{\frac{4}{\tau}\int_{0}^{\tau }|v_{1}^{*}B_{u}+v_{2}^{*
}D_{u}|^{2}du, \
4 \Big\{\int_{0}^{\tau }|v_{1}^{*}B_{u}+v_{2}^{*
}D_{u}|^{2}du\Big\}^{\frac{1}{5}}  \Big( \|v_{1}^{*}B_{.}+v_{2}^{* }D_{.}\|^{2}_{\frac{1}{4}}\Big) ^{\frac{4}{5}}\Bigg\}.
\end{split}
\end{eqnarray*}
Thus
\begin{eqnarray} \label{22-1}
\nonumber &&\mathbb{P}\Big\{F\cap \big\{\sup_{s\in \lbrack 0,\tau ]}|v_{1}^{*}B_{s}+v_{2}^{*
}D_{s}|^{2}>\epsilon ^{\frac{q^{3j_{0}+6}}{10}}\big\}\Big\} \\
&& \leq \mathbb{P}\Big\{\|v_{1}^{*}B_{.}+v_{2}^{*}D_{.}\|^{2}_{\frac{1}{4}%
}>\frac{1}{4^{\frac{5}{4}}C_0^{\frac{1}{4}}}\epsilon ^{-\frac{q^{3j_{0}+6}}{8}}\Big\} +\mathbb{P}\left(\tau<4C_0 \epsilon^{\frac{9}{10}q^{3j_0+6}} \right).
\end{eqnarray}
Due to
 $(\ref{22-1})$,  Lemma $\ref{2-290}$   and Lemma $\ref{29}$,   for any $p>0,$  there exists a constant $C(p,T,x,y,q)$ such that
\begin{eqnarray*}
\mathbb{P}\Big\{F\cap \big\{\sup_{s\in \lbrack 0,\tau
]}|v_{1}^{*}B_{s}+v_{2}^{*}D_{s}|^{2}>\epsilon ^{\frac{%
q^{3j_{0}+6}}{10}}\big\}\Big\}\leq C(p,T,x,y,q)\epsilon ^{p}, ~\forall  \epsilon >0.
\end{eqnarray*}
\end{proof}

\begin{lemma}
\label{20}  Let the  Hypothesis $\ref{2-18}$ hold,   then for any
$p>0$,  there exists positive constant $C(p,T,x,y,q)$  such that
\begin{eqnarray*}
\mathbb{P}\bigg\{\sum_{j=1}^{d}\int_{0}^{T}|(v_{1}^{*
}B_{s}+v_{2}^{*}D_{s})b_{j}|^{2}ds &\leq &\epsilon
^{q^{3j_{0}+6}},\int_{0}^{\tau}|(v_{1}^{*}A_{s}+v_{2}^{*
}C_{s})\nabla _{y}a_{1}|^{2}ds>\epsilon ^{q^{3j_{0}}}\bigg\} \\
&\leq &C(p,T,x,y,q)\epsilon ^{p},~\forall  \epsilon >0.
\end{eqnarray*}
\end{lemma}

\begin{proof}
From $(\ref{2-5})$,
\begin{eqnarray*}
\begin{split}
 &d(v_{1}^{*}B_{s}+v_{2}^{*}D_{s})
\\ &=-(v_{1}^{*}B_{s}+v_{2}^{*
}D_{s})\nabla_y a_2ds-(v_{1}^{*}A_{s}+v_{2}^{*
}C_{s})\nabla _{y}a_{1}ds \\
 & \ \ \ \ -\sum_{j=1}^d (v_{1}^{*}B_{s}+v_{2}^{*}D_{s})\nabla _{y}b_{j}
dW_{j}(t)+\sum_{j=1}^d (v_{1}^{*}B_{s}+v_{2}^{*}D_{s})\nabla _{y}b_{j}\nabla
_{y}b_{j}ds.
\end{split}
\end{eqnarray*}
From $\det(b(x,y)b^{*}(x,y)) \neq 0$ and the definition of
$\tau $,  if
\begin{eqnarray*}
\sum_{j=1}^{d}\int_{0}^{T}|(v_{1}^{*}B_{s}+v_{2}^{*
}D_{s})b_{j}|^{2}(\omega)ds\leq \epsilon ^{q^{3j_{0}+6}},
\end{eqnarray*}%
then there exists   a constant  $C=C(x,y)$  such that
\begin{eqnarray}\label{2-7}
\int_{0}^{\tau}|v_{1}^{* }B_{s}+v_{2}^{* }D_{s}|^{2}(\omega)ds\leq
C\epsilon ^{q^{3j_{0}+6}}.
\end{eqnarray}
Define
\begin{equation}
\label{2-110}
\begin{split}
\tilde{y}(s): &=(v_{1}^{* }B_{s}+v_{2}^{*
}D_{s})+\int_{0}^{s}(v_{1}^{* }B_{u}+v_{2}^{* }D_{u})%
\nabla_y a_2du   -\sum_{j=1}^{d}\int_{0}^{s}(v_{1}^{* }B_{u}+v_{2}^{* }D_{u})\nabla _{y}b_{j}\nabla
_{y}b_{j}du,
\end{split}
\end{equation}%
then
\begin{eqnarray*}
d\tilde{y}(s)=-(v_{1}^{* }A_{s}+v_{2}^{* }C_{s})\nabla
_{y}a_{1}ds-\sum_{j=1}^{d}(v_{1}^{* }B_{s}+v_{2}^{* }D_{s})\nabla _{y}b_{j}
dW_{j}(s).
\end{eqnarray*}%
Due to   H\"{o}lder inequality, $(\ref{2-7})$  and  $(\ref{2-110})$, there exists a constant $C(T,x,y)$ such that
\begin{eqnarray*}
\int_{0}^{\tau }|\tilde{y}(s)|^{2}ds &\leq &C(T,x,y)\int_{0}^{\tau }|v_{1}^{*
}B_{s}+v_{2}^{* }D_{s}|^{2}ds
\end{eqnarray*}%
This implies that
\begin{eqnarray*}
\Bigg\{ \sum_{j=1}^{d}\int_{0}^{T}|(v_{1}^{* }B_{s}+v_{2}^{*
}D_{s})b_{j}|^{2}ds\leq \epsilon ^{q^{3j_{0}+6}},~\int_{0}^{\tau
}|v_{1}^{* }A_{s}+v_{2}^{* }C_{s})\nabla _{y}a_{1}||^{2}ds>
\epsilon ^{q^{3j_{0}}}\Bigg\}\\
\subseteq
\begin{split}
&\Bigg\{\int_{0}^{\tau }|\tilde{y}(s)|^{2}ds\leq C(T,x,y)\epsilon ^{q^{3j_{0}+6}},\
\int_{0}^{\tau }|v_{1}^{* }A_{s}+v_{2}^{* }C_{s})\nabla
_{y}a_{1}|^{2}ds\geq \epsilon ^{q^{3j_{0}}}\Bigg\}.
\end{split}
\end{eqnarray*}
The probability of the  above event can be estimated  by   Lemma $\ref{23}$
and Lemma $\ref{2-290}$.
\end{proof}

\begin{lemma}
\label{21}Let the  Hypothesis $\ref{2-18}$ hold, then for any $p>0$,
there exists constants $C=C(p,T,x,y,q), ~\epsilon_0=\epsilon_0(q,x,y)$  such that for $j=1,\cdots,j_0-1,$
\begin{eqnarray*}
\mathbb{P}\{F\cap E_{j}\cap E_{j+1}^{c}\}\leq C(p,T,x,y,q)\epsilon ^{p}, ~\forall  \epsilon \leq \epsilon_0.
\end{eqnarray*}
\end{lemma}
\begin{proof}
For any $K\in \mathcal{A}_{j}$,
by calculating,
\begin{eqnarray*}
\begin{split}
&d(v_{1}^{* }A_{s}+v_{2}^{* }C_{s})K(x_{s},y_{s}) \\
&=-\sum_{i=1}^{d}\langle (v_{1}^{* }B_{s}
 +v_{2}^{* }D_{s})\nabla
_{x}b_{i},\nabla _{y}K b_{i}\rangle ds+\sum_{i=1}^d(v_{1}^{* }B_{s}+v_{2}^{* }D_{s})\nabla
_{y}b_{i}\nabla _{x}b_{i}K(x_{s},y_{s})ds \\
&\ \ \  +\sum_{i=1}^{d}\Big((v_{1}^{* }A_{s}+v_{2}^{* }C_{s})\nabla
_{y}K(x_{s},y_{s})b_{i}-(v_{1}^{* }B_{s}+v_{2}^{* }D_{s})\nabla
_{x}b_{i}K(x_{s},y_{s})\Big) \cdot dW_{i}(s) \\
&\ \ \ +(v_{1}^{* }A_{s}+v_{2}^{* }C_{s})\nabla
_{y}K(x_{s},y_{s})a_{2}(x_{s},y_{s})ds-(v_{1}^{* }B_{s}+v_{2}^{*
}D_{s})\nabla _{x}a_{2}K(x_{s},y_{s})ds \\
&\ \ \ +(v_{1}^{* }A_{s}+v_{2}^{* }C_{s})\big(-\nabla
_{x}a_{1}(x_{s},y_{s})K(x_{s},y_{s})+\nabla
_{x}K(x_{s},y_{s})a_{1}(x_{s},y_{s})\big) ds \\
&\ \ \ +\frac{1}{2}(v_{1}^{* }A_{s}+v_{2}^{* }C_{s})\sum_{i=1}^{d}\big(
\nabla _{y}(\nabla _{y}K\cdot b_{i})b_{i}\big) ds,
\end{split}
\end{eqnarray*}
and
\begin{eqnarray*}
\begin{split}
&\mathbb{P}(F\cap E_{j}\cap E_{j+1}^{c})
\\ &=\mathbb{P}\left( F\cap \sum_{K\in \mathcal{A}%
_{j}}\int_{0}^{\tau }\big|(v_{1}^{* }A_{s}+v_{2}^{*
}C_{s})K(x_{s},y_{s})\big|^{2}ds\leq \epsilon ^{q^{3j_{0}+3-3j}},\right.  \\
&\ \ \ \ \ \ \ \ \ \ \ \ \ \ \ \left. \sum_{K\in \mathcal{A}_{j+1}}\int_{0}^{\tau }\big|(v_{1}^{*
}A_{s}+v_{2}^{* }C_{s})K(x_{s},y_{s})\big|^{2}ds>\epsilon
^{q^{3j_{0}-3j}}\right)  \\
&\leq \sum_{K\in \mathcal{A}_{j}}\mathbb{P}\Bigg(F\cap \int_{0}^{\tau
}\big|(v_{1}^{* }A_{s}+v_{2}^{* }C_{s})K(x_{s},y_{s})\big|^{2}ds\leq
\epsilon ^{q^{3j_{0}+3-3j}}, \\
&\ \ \ \ \ \ \ \ \ \ \ \ \ \  \int_{0}^{\tau }\big|(v_{1}^{* }A_{s}+v_{2}^{*
}C_{s}) \big( -\nabla_x a_{1} (x_{s},y_{s})K(x_{s},y_{s})+%
\nabla _{x}K\cdot a_{1}(x_{s},y_{s})\big) \big|^{2}ds
  \\
&\ \ \ \ \ \ \ \ \ \ \ \ \ \  \ \  \ \   \ \ \ \ \  \ \  \ \ \ \  \ \  \ \ \ \  \   \ \ \  +\int_{0}^{\tau }\big|(v_{1}^{* }A_{s}+v_{2}^{* }C_{s})\nabla
_{y}K(x_{s},y_{s})\big|^{2}ds\geq \frac{\epsilon
^{q^{3j_{0}-3j}}}{n(j)}\Bigg),
\end{split}
\end{eqnarray*}%
where $ n(j)$ denotes the cardinality of the set  $\mathcal{A}_j$.

It is not difficult to prove   there exists a  constant $\epsilon_0=\epsilon_0(q,x,y)$, such that when
$\epsilon<\epsilon_0$
\begin{align*}
 F&\cap \left\{\int_{0}^{\tau }|(v_{1}^{* }A_{s}+v_{2}^{*
}C_{s})K(x_{s},y_{s})|^{2}ds\leq \epsilon ^{q^{3j_{0}+3-3j}}\right\}\\
& \cap \left\{ \int_{0}^{\tau }\big|(v_{1}^{* }A_{s}+v_{2}^{* }C_{s})\nabla
_{y}K(x_{s},y_{s})\big|^{2}ds+\right.
 \\
&\ \ \ \  \int_{0}^{\tau }\big|(v_{1}^{* }A_{s}+v_{2}^{* }C_{s})\big(
-\nabla_xa_{1}(x_{s},y_{s})K(x_{s},y_{s})+\nabla _{x}K\cdot
a_{1}(x_{s},y_{s})\big) \big|^{2}ds
\\
&\ \ \ \ \ \ \ \left. \geq \frac{\epsilon ^{q^{3j_{0}-3j}}}{n(j)}\right\}\subseteq  B_{1}\cup B_{2},
\\ & \text{and} \ B_1 \subseteq B_1^{'},~ B_{2}\subseteq B_{3}\cup B_{4},~B_4 \subseteq \cup _{i=1}^{d}B_{4i}^{^{\prime }}, ~B_{4i}^{^{\prime }} \subseteq B_{4i}^{^{\prime \prime }},
\end{align*}%
here,
\begin{align*}
B_{1} &=F\cap \left\{ \int_{0}^{\tau }\big|(v_{1}^{* }A_{s}+v_{2}^{*
}C_{s})K(x_{s},y_{s})\big|^{2}ds\leq \epsilon ^{q^{3j_{0}+3-3j}},\right.  \\
&\ \ \ \ \ \ \ \ \ \ \ \ \ \  \left. \sum_{i=1}^{d}\int_{0}^{\tau }\big|(v_{1}^{* }A_{s}+v_{2}^{*
}C_{s})\nabla _{y}K(x_{s},y_{s})b_{i}\big|^{2}ds>\epsilon
^{q^{3j_{0}+2-3j}}\right\},
\\
B_{1}^{^{\prime }} &=F\cap \left\{\int_{0}^{\tau }\big|(v_{1}^{*
}A_{s}+v_{2}^{* }C_{s})K(x_{s},y_{s})\big|^{2}ds\leq \epsilon
^{q^{3j_{0}+3-3j}}\right\} \\
&\ \ \ \cap\left\{\sum_{i=1}^{d}\int_{0}^{\tau }\bigg(\bigg|(v_{1}^{* }A_{s}+v_{2}^{*
}C_{s})\nabla _{y}K(x_{s},y_{s})b_{i} \right. \\
&\ \ \ \ \ \ \ \ \  \ \ \  \ \ \ \ \ \ \  \ \left. -(v_{1}^{* }B_{s}+v_{2}^{* }D_{s})\nabla
_{x}b_{i}K(x_{s},y_{s})\bigg|^{2}\bigg)ds>\frac{1}{4}\epsilon ^{q^{3j_{0}+2-3j}}\right\},
\\
B_{2}&=F\cap \left\{ \int_{0}^{\tau }\big|(v_{1}^{* }A_{s}+v_{2}^{*
}C_{s})K(x_{s},y_{s})\big|^{2}ds\leq \epsilon ^{q^{3j_{0}+3-3j}}\right\}  \\
&\ \ \ \cap \left\{ \int_{0}^{\tau }\big|(v_{1}^{* }A_{s}+v_{2}^{* }C_{s})\left(
-\nabla _{x}a_{1}\cdot K+\nabla _{x}K\cdot a_{1}\right) \big|^{2}ds\geq \frac{%
\epsilon ^{q^{3j_{0}-3j}}}{{2n(j)}}\right\}  \\
&\ \ \ \cap \left\{ \sum_{i=1}^{d}\int_{0}^{\tau }\big|(v_{1}^{*
}A_{s}+v_{2}^{* }C_{s})\nabla _{y}K(x_{s},y_{s})b_{i}\big|^{2}ds\leq
\epsilon ^{q^{3j_{0}+2-3j}}\right\},
\\
B_{3}&=F\cap \left\{ \int_{0}^{\tau }\big|(v_{1}^{* }A_{s}+v_{2}^{*
}C_{s})K(x_{s},y_{s})\big|^{2}ds\leq \epsilon ^{q^{3j_{0}+3-3j}}\right\} \\
&\ \ \ \cap \left\{ \int_{0}^{\tau }\big|(v_{1}^{* }A_{s}+v_{2}^{* }C_{s})\left(
-\nabla _{x}a_{1}\cdot K+\nabla _{x}K\cdot a_{1}\right) \big|^{2}ds\geq \frac{%
\epsilon ^{q^{3j_{0}-3j}}}{{2n(j)}}\right\} \\
&\ \ \ \cap \left\{ \sum_{i=1}^{d}\int_{0}^{\tau }\big|(v_{1}^{*
}A_{s}+v_{2}^{* }C_{s})\nabla _{y}K(x_{s},y_{s})b_{i}\big|^{2}ds\leq
\epsilon ^{q^{3j_{0}+2-3j}}\right\}
\\
&\ \ \ \cap \left\{ \sum_{i=1}^{d}\int_{0}^{\tau }\big|(v_{1}^{*
}A_{s}+v_{2}^{* }C_{s})\nabla _{y}\left( \nabla
_{y}K(x_{s},y_{s})b_{i}\right) b_{i}\big|^{2}ds\leq \epsilon
^{q^{3j_{0}-3j+1}}\right\} ,
\\
B_{4}&=F\cap \left\{ \int_{0}^{\tau }\big|(v_{1}^{* }A_{s}+v_{2}^{*
}C_{s})K(x_{s},y_{s})\big|^{2}ds\leq \epsilon ^{q^{3j_{0}+3-3j}}\right\}
\\
&\ \ \ \cap \left\{ \int_{0}^{\tau }\big|(v_{1}^{* }A_{s}+v_{2}^{* }C_{s})\left(
-\nabla _{x}a_{1}\cdot K+\nabla _{x}K\cdot a_{1}\right) \big|^{2}ds\geq \frac{%
\epsilon ^{q^{3j_{0}-3j}}}{{2n(j)}}\right\} \\
&\ \ \ \cap \left\{ \sum_{i=1}^{d}\int_{0}^{\tau }\big|(v_{1}^{*
}A_{s}+v_{2}^{* }C_{s})\nabla _{y}K(x_{s},y_{s})b_{i}\big|^{2}ds\leq
\epsilon ^{q^{3j_{0}+2-3j}}\right\} \\
&\ \ \ \cap \left\{ \sum_{i=1}^{d}\int_{0}^{\tau }\big|(v_{1}^{*
}A_{s}+v_{2}^{* }C_{s})\nabla _{y}\left( \nabla
_{y}K(x_{s},y_{s})b_{i}\right) b_{i}\big|^{2}ds>\epsilon
^{q^{3j_{0}-3j+1}}\right\} ,
\\
B_{4i}^{^{\prime }} &=F\cap \left\{ \int_{0}^{\tau }\big|(v_{1}^{*
}A_{s}+v_{2}^{* }C_{s})\nabla _{y}K(x_{s},y_{s})b_{i}\big|^{2}ds\leq
\epsilon ^{q^{3j_{0}+2-3j}}\right\} \\
&\ \ \ \cap \left\{ \int_{0}^{\tau }\big|(v_{1}^{* }A_{s}+v_{2}^{*
}C_{s})\nabla _{y}\big( \nabla _{y}K(x_{s},y_{s})b_{i}\big) b_{i}\big|^{2}ds>%
\frac{1}{d}\epsilon ^{q^{3j_{0}-3j+1}}\right\},
\\
B_{4i}^{^{\prime \prime  }} &=F\cap \left\{ \int_{0}^{\tau }\big|(v_{1}^{*
}A_{s}+v_{2}^{* }C_{s})\nabla _{y}K(x_{s},y_{s})b_{i}\big|^{2}ds\leq
\epsilon ^{q^{3j_{0}+2-3j}}\right\} \\
&\ \ \ \cap \left\{ \sum_{k=1}^{d}\int_{0}^{\tau }\big|(v_{1}^{*
}A_{s}+v_{2}^{* }C_{s})\nabla _{y}\left( \nabla
_{y}K(x_{s},y_{s})b_{i}\right) b_{k}\big|^{2}ds>\frac{1}{d}\epsilon
^{q^{3j_{0}-3j+1}}\right\} .
\end{align*}
From  Lemma $\ref{25}$,
\begin{align*}
 \mathbb{P}(B_{1})\leq  \mathbb{P}(B_{1}^{^{\prime }} ) \leq C(p,T,x,y) \epsilon^p,~ \forall \epsilon>0.
\end{align*}
The  estimation of $\mathbb{P}(B_{4i}^{^{\prime \prime }})$  is similar to the estimation  of $\mathbb{P}(B_1)$.
For $\mathbb{P}(B_{3})$, define
\begin{align*}\label{2-80}
\tilde{y}(s):&=(v_{1}^{* }A_{s}+v_{2}^{* }C_{s})K(x_{s},y_{s})
 -\frac{1}{2}%
\sum_{i=1}^{d}\int_{0}^{s}(v_{1}^{* }A_{r}+v_{2}^{* }C_{r})\nabla_y \left(
\nabla _{y}K(x_{r},y_{r})b_{i}\right) b_{i}dr \\
&-\int_{0}^{s}(v_{1}^{* }A_{r}+v_{2}^{* }C_{r})\nabla
_{y}K(x_{r},y_{r})a_{2}(x_{r},y_{r})dr
+\int_{0}^{s}(v_{1}^{*
}B_{r}+v_{2}^{* }D_{r})\nabla _{x}a_{2}K(x_{r},y_{r})dr \\
&-\sum_{i=1}^{d}\int_{0}^{s}(v_{1}^{* }B_{r}+v_{2}^{* }D_{r})\nabla
_{y}b_{i}\nabla _{x}b_{i}K(x_{r},y_{r})dr +\sum_{i=1}^{d}\int_{0}^{s}\langle (v_{1}^{* }B_{r}+v_{2}^{*
}D_{r})\nabla _{x}b_{i},\nabla _{y}Kb_{i}\rangle dr,
\end{align*}%
then from  the equation  $(v_{1}^{* }A_{s}+v_{2}^{* }C_{s})K(x_s,y_s)$  satisfied, we know
\begin{align*}
d\tilde{y}(s) &=\sum_{i=1}^{d}\Big[(v_{1}^{* }A_{s}+v_{2}^{* }C_{s})\nabla
_{y}K(x_{s},y_{s})b_{i}+ (v_{1}^{*
}B_{s}+v_{2}^{* }D_{s})\nabla _{x}b_{i}K(x_{s},y_{s})\Big]dW_{i}(s)
 \\
&\ \  \ +(v_{1}^{* }A_{s}+v_{2}^{* }C_{s})\big( -\nabla
_{x}a_{1}K+\nabla
_{x}Ka_{1}\big)(x_s,y_s) ds.
\end{align*}%
If $\omega \in B_3$,  from the definitions  of $\tilde{y}(s)$ and $\tau$, there
exists a constant $C(p,T,x,y)$   such that
$$
\int_{0}^{\tau }\tilde{y}(s)^{2}ds\leq C(p,T,x,y)\epsilon ^{q^{3j_{0}+1-3j}}.
$$
Thus $B_{3}$ is a subset of
\begin{eqnarray}\label{2-10}
 \begin{split}
&\ \ \ \ \ \  \left\{ \int_{0}^{\tau }\tilde{y}(s)^{2}ds\leq  C(p,T,x,y)  \epsilon
^{q^{3j_{0}+1-3j}},\right. \\
& \ \  \ \ \ \ \  \left. \int_{0}^{\tau }\big|(v_{1}^{* }A_{s}+v_{2}^{*
}C_{s})\left( -\nabla _{x}a_{1}\cdot K+\nabla _{x}K\cdot a_{1}\right)
\big|^{2}ds\geq \frac{\epsilon ^{q^{3j_{0}-3j}}}{{2n(j)}}\right\}.
\end{split}
\end{eqnarray}%
The estimate   of the above set is from  Lemma $\ref{23}$ and Lemma $ \ref{2-290}$. This finishes
the proof of the Lemma 2.13.
\end{proof}

\begin{lemma}\label{2-1-6}
Let the  Hypothesis $\ref{2-18}$ hold, then there  exists a  constant $ \epsilon_0=\epsilon_0(q,x,y)$ such that
 \begin{eqnarray*}
 E\cap \{\tau \geq \epsilon ^{q }\}\cap \big\{\sup_{s\in
\lbrack 0,\tau ]}|v_{1}^{*}B_{s}+v_{2}^{*}D_{s}|\leq \epsilon ^{%
\frac{q^{3j_{0}+6}}{10}}\big\}=\emptyset, ~\forall \epsilon<\epsilon_0.
\end{eqnarray*}
\end{lemma}
\begin{proof}
If $\omega  \in E\cap \{\tau \geq \epsilon ^{q }\}\cap \big\{\sup\limits_{s\in
\lbrack 0,\tau ]}|v_{1}^{*}B_{s}+v_{2}^{*}D_{s}|\leq \epsilon ^{%
\frac{q^{3j_{0}+6}}{10}}\big\}$, from $(\ref{2-2-3})$,  for some $c>0,$
\begin{eqnarray}\label{2-2-10}
\begin{split}
 &\sum_{j=1}^{j_{0}}\sum_{V\in \mathcal{A}_{j}}\int_{0}^{\tau
}\big|(v_{1}^{*}A_{s}+v_{2}^{*}C_{s})V(x_{s},y_{s})(\omega)\big|^{2}ds\\
&=\int_{0}^{\tau
}\sum_{j=1}^{j_{0}}\sum_{V\in \mathcal{A}_{j}}\Big(\frac{(v_{1}^{*}A_{s}+v_{2}^{*}C_{s})}{ \big|v_{1}^{*}A_{s}+v_{2}^{*
}C_{s}\big|  }V(x_{s},y_{s})(\omega)\Big)^2
 \cdot   \left|v_{1}^{*}A_{s}+v_{2}^{*
}C_{s}\right|^{2}ds
\\  & \geq c \int_{0}^{\tau }\big|v_{1}^{*}A_{s}+v_{2}^{*
}C_{s}\big|^{2}ds.
\end{split}
\end{eqnarray}
For $\omega \in
\big\{\sup_{s\in
\lbrack 0,\tau ]}|v_{1}^{*}B_{s}+v_{2}^{*}D_{s}|\leq \epsilon ^{%
\frac{q^{3j_{0}+6}}{10}}\big\}$,
let $s=0$,
\begin{eqnarray}\label{2-2-15}
|v_{2}|\leq \epsilon ^{\frac{q^{3j_{0}+6}}{10}}\leq \frac{1}{100},~|v_1|=\sqrt{1-|v_2|^2}>\frac{9}{10}.
\end{eqnarray}
Due to $(\ref{2-2-15})$ and   the fact  when  $s \leq \tau
$, $\|A_s-I_m\|\leq \frac{1}{100},~\|C_s\|\leq \frac{1}{100}$,
\begin{eqnarray} \label{2-2-20}
\int_{0}^{\tau }\big|v_{1}^{*}A_{s}+v_{2}^{*
}C_{s}\big|^{2}ds\geq \frac{1}{8}\tau \geq \frac{1}{8}\epsilon ^q.
\end{eqnarray}
From $(\ref{2-2-10})$ and $(\ref{2-2-20})$,
\begin{eqnarray} \label{2-2-21}
\sum_{j=1}^{j_{0}}\sum_{V\in \mathcal{A}_{j}}\int_{0}^{\tau
}\big|(v_{1}^{*}A_{s}+v_{2}^{*}C_{s})V(x_{s},y_{s})(\omega)\big|^{2}ds \geq  \frac{c}{8}\epsilon ^q.
\end{eqnarray}
In the following part, we will prove this is impossible when $\epsilon$ is small enough.
Set $\epsilon_0(q,x,y)$ such that when   $\epsilon<\epsilon_0(q,x,y),$
\begin{eqnarray*}
\epsilon ^{\frac{q^{3j_{0}+6}}{10}} \leq \frac{1}{100}, ~\sum_{j=1}^{j_{0}}\epsilon ^{q^{3j_{0}+3-3j}}<\frac{c}{8}\epsilon
^{q }.
\end{eqnarray*}
For  $\omega \in E \subseteq  E_{j}$,  then
\begin{eqnarray*}
\sum_{K(x,y)\in \mathcal{A}_{j}}\int_{0}^{\tau }\big|(v_{1}^{*
}A_{s}+v_{2}^{*}C_{s})K(x_{s},y_{s})(\omega)\big|^{2}ds\leq \epsilon
^{q^{3j_{0}+3-3j}},
\end{eqnarray*}
so  when   $\epsilon<\epsilon_0,$
\begin{eqnarray*}
\sum_{j=1}^{j_{0}}\sum_{V\in \mathcal{A}_{j}}\int_{0}^{\tau
}\big|(v_{1}^{* }A_{s}+v_{2}^{*}C_{s})V(x_{s},y_{s})(\omega)\big|^{2}ds\leq
\sum_{j=1}^{j_{0}}\epsilon ^{q^{3j_{0}+3-3j}}<\frac{c}{8}\epsilon
^{q },
\end{eqnarray*}
this contradict with  $(\ref{2-2-21})$.

Thus $E\cap \{\tau \geq
\epsilon ^{ }\}\cap \big\{\sup_{s\in \lbrack 0,\tau
]}|v_{1}^{*}B_{s}+v_{2}^{*}D_{s}|^{2}\leq \epsilon ^{\frac{%
q^{3j_{0}+6}}{10}}\big\}=\emptyset $  when  $\epsilon<\epsilon_0$.
\end{proof}

We are now in a position to give
\begin{proof}{\bf The proof of Theorem $\ref{15}$:}
Since
\begin{eqnarray}\label{2-3-2}
M_{T}=J_{T} \tilde{M}_{T}J_{T}^{*},
\end{eqnarray}
where
\begin{eqnarray} \label{2-3-1}
&&  \tilde{M}_{T}=\int_{0}^{T}J_{s}^{-1}\Big(
\begin{matrix}
 0 \\
 b(x_{s},y_{s})%
\end{matrix}%
\Big) \Big(
\begin{matrix}
0 \\
 b(x_{s},y_{s})%
\end{matrix}%
\Big) ^{*}(J_{s}^{-1})^{*}ds,
\end{eqnarray}
 we only need  to prove the $L^{p}$ integrability of
$\det(\tilde{M}_T^{-1})$. For this  purpose, we need to prove for any
$p>0$, there exists constant $C(p)$, such that
\begin{eqnarray}\label{2-106}
\sup_{|v|=1}\mathbb{P}\{v^{*}\tilde{M}_{T}v\leq \epsilon \}\leq C(p)\epsilon
^{p},\ \ \ \forall \epsilon >0.
\end{eqnarray}
It is easy to check that  $(\ref{2-106})$ is
equivalent to for any $p>0,v \in \mathbb{R}^{m+n},|v|=1$, there exists positive  constants
$\epsilon_0(p),~C(p) $ such that
\begin{eqnarray}\label{2-301}
\mathbb{P}\{v^{*}\tilde{M}_{T}v\leq \epsilon \}\leq C(p)\epsilon
^{p},\ \ \ \forall \epsilon \leq \epsilon_0(p).
\end{eqnarray}
For  $~v=(v_{1}^*,v_{2}^*)^{*} \in \mathbb{R}^{m}\times \mathbb{R}^{n}$,  $J^{-1}(s)=\left(
\begin{array}{cc}
A_{s} & B_{s} \\
C_{s} & D_{s}%
\end{array}%
\right)$  and due to    $(\ref{2-3-1})$,
\begin{eqnarray*}
v^{*}\tilde{M}_{T}v=\sum_{j=1}^{d}\int_{0}^{T}|(v_{1}^{*
}B_{s}+v_{2}^{*}D_{s})b_{j}|^{2}ds.
\end{eqnarray*}
Here, we recall the definitions  of $E,F,E_j,\tau$ which are given  in the beginnings  of this subsection.
Then $(\ref{2-301})$  is equivalent to for any $p>0$ and $v \in \mathbb{R}^{m+n},|v|=1$, there exists constants  $C(p)$ and $\epsilon_0(p)$ such that
$$
\mathbb{P}(F) \leq C(p) \epsilon^p,~ \forall \epsilon \leq \epsilon_0(p).
$$
For
\begin{eqnarray*}
F\subseteq (F\cap E_{1}^{c})\cup (F\cap E_{1}\cap E_{2}^{c})\cup (F\cap
E_{2}\cap E_{3}^{c})\cup \cdots \cup (F\cap E_{j_{0}-1}\cap
E_{j_{0}}^{c})\cup E,
\end{eqnarray*}%
so
\begin{eqnarray}\label{2-402}
\mathbb{P}(F)\leq \mathbb{P}(E)+\sum_{j=1}^{j=j_{0}-1}\mathbb{P}({F\cap E_{j}\cap E_{j+1}^{c}})+\mathbb{P}(F\cap
E_{1}^{c}).
\end{eqnarray}
From Lemma $\ref{20}$ and Lemma $\ref{21}$, for any $p>0$ and $v \in \mathbb{R}^{m+n},|v|=1$,
there exists positive  constants $C(p,T,x,y,q), \epsilon_0(q,x,y)$ such that for any $\epsilon\leq \epsilon_0(q,x,y)$,
\begin{eqnarray}\label{2-401}
\sum_{j=1}^{j=j_{0}-1}\mathbb{P}({F\cap E_{j}\cap E_{j+1}^{c}})+\mathbb{P}(F\cap
E_{1}^{c})\leq C(p,T,x,y,q)\epsilon ^{p}.
\end{eqnarray}%
For estimating $\mathbb{P}(E)$,  we note that
\begin{equation}\label{2-303}
\begin{split}
 \mathbb{P}(E) &\leq \mathbb{P}\big(E\cap \{\tau \geq \epsilon ^{q }\}\big)+\mathbb{P}(\tau <\epsilon
^{q }) \\
&\leq  \mathbb{P}\Big(F\cap \{\tau \geq \epsilon ^{q}\}\cap \{\sup_{s\in \lbrack 0,\tau
]}|v_{1}^{*}B_{s}+v_{2}^{*}D_{s}|>\epsilon ^{\frac{q^{3j_{0}+6}}{%
10}}\}\Big)
\\   & \ \  +\mathbb{P}\Big(E\cap \{\tau \geq \epsilon ^{q }\}\cap \big\{\sup_{s\in \lbrack
0,\tau ]}|v_{1}^{*}B_{s}+v_{2}^{*}D_{s}|\leq \epsilon ^{\frac{%
q^{3j_{0}+6}}{10}}\big\}\Big)
\\ &  \ \  + \mathbb{P}\left(\tau <\epsilon ^{q }\right).
\end{split}
\end{equation}
So, due to  Lemma $\ref{19}$, Lemma $\ref{2-1-6}$ and  Lemma  $ \ref{2-290}$,
there  exists constants  $C(p,T,x,y,q)$ and $\epsilon_0=\epsilon_0(q,x,y) $ such that
\begin{eqnarray}\label{2-403}
\mathbb{P}(E) \leq C(p,T,x,y,q) \epsilon^p,~ \forall \epsilon \leq \epsilon_0.
\end{eqnarray}

So from $(\ref{2-402})(\ref{2-401})(\ref{2-403})$, we know for any $p>0$ and $v \in \mathbb{R}^{m+n},|v|=1$, there
exists constants $C(p,T,x,y,q)$  and  $\epsilon_0(q,x,y) $  such that
$$
\mathbb{P}(F)\leq C(p,T,x,y,q) \epsilon^p, ~\forall  \epsilon \leq \epsilon_0(q,x,y).
$$
 Since $T,x,y,q$  are all fixed, this theorem has been proved.
\end{proof}

\section{Gradient Estimate}
In this section, we give  a gradient estimate. The Hypothesis and  main Theorem in this section  is
\begin{hypothesis}
\label{100-1} There exists $j_0\in \mathbb{N}$ and $R>0$ such that
\begin{description}
  \item[(\romannumeral1)]  $a_1 \in C_b^2(\mathbb{R}^m\times \mathbb{R}^n;\mathbb{R}^m)
 \cap C^{j_0+2}(\mathbb{R}^m\times \mathbb{R}^n;\mathbb{R}^m),~a_2\in C_b^2(\mathbb{R}^m
 \times \mathbb{R}^n;\mathbb{R}^n);$
  \item[(\romannumeral2)] $b\in C_b^2(\mathbb{R}^m\times \mathbb{R}^n;\mathbb{R}^n \times \mathbb{R}^d)$,  $\det(b(x,y)\cdot b^{*}(x,y))\neq 0$, $\forall (x,y) \in \mathbb{R}^m\times \mathbb{R}^n$ with $|(x,y)|\leq R;$
  \item[(\romannumeral3)]   $\forall (x,y) \in \mathbb{R}^m\times \mathbb{R}^n,  |(x,y)|\leq R,$
the vector space spanned by  $\cup_{k=1}^{j_0} \mathcal{A}_{k}$ at point $(x,y)$ has dimension $m$.
\end{description}
\end{hypothesis}
\begin{thm}
$\label{2-39}$ Let the  Hypothesis $\ref{100-1}$ hold, then for any $t>0$,  then there
exists a constant  $C=C\left(R,t\right)$ such that for any  $f \in
C_b^{1}(\mathbb{R}^m \times \mathbb{R}^n,\mathbb{R})$, $(x,y)\in \mathbb{R}^m \times  \mathbb{R}^n$ with  $ |(x,y)|\leq R$,
\begin{align*}
|  \nabla P_{t}f\left( x,y\right) |  \leq
C\left( R, t\right) \| f\|_{\infty }.
\end{align*}
\end{thm}

In order to prove this Theorem, we need  the following Lemmas.
These Lemmas give  some estimations  of   $J_t$, $J_t^{-1}$, $(x_t,y_t)$ and $M_t$. Especially,  we give a uniform estimation of  $M_t$ in Lemma  $\ref{100-16}$.
In the end of this section, we give the proof of Theorem $\ref{2-39}$. The  method to prove Theorem $\ref{2-39}$  is standard.

Before all of  this, we introduce  some notations. Let $D(x_t,y_t)$ denote the Malliavin derivative of $(x_t,y_t)$ and $H=L^2([0,\infty),ds)$.
$\delta$ denotes  the divergence operator.

\begin{lemma} \label{100-14}
Let the   Hypothesis $\ref{100-1}$ hold, then for any $T,p>0$,
\begin{eqnarray}
\label{100-10}&&\sup_{|(x,y)|\leq R}\mathbb{E}_{x,y}\Big\{ \sup_{t \in [0,T]}|(x_t,y_t)|^p \Big\}<\infty,
\\ \label{100-11}&&
\sup_{|(x,y)|\leq R}\mathbb{E}_{x,y}\Big\{ \sup_{t \in [0,T]}\|J_t\|^p\Big\} <\infty,
\\  \label{100-12}&& \sup_{|(x,y)|\leq R}\mathbb{E}_{x,y}\Big\{ \sup_{t \in [0,T]}
\|J_t^{-1}\|^p \Big\}<\infty,
\\  \label{100-20}&& \sup_{|(x,y)|\leq R}\mathbb{E}_{x,y} \|M_T\|^p <\infty,
\end{eqnarray}
\end{lemma}
\begin{proof}
For any$(x,y)$  fixed ,  $\mathbb{E}_{x,y}\big\{\sup_{s \in [0,T]}|(x_s,y_s)|^p \big\} <\infty$,  and  $h(x,y):=\mathbb{E}_{x,y}\big\{ \sup_{t \in [0,T]}|(x_t,y_t)|^p \big\}$ is continuous w.r.t $(x,y)$, so $(\ref{100-10})$ holds.

For any  $p>2$, set $f(t)=\mathbb{E}_{x,y}\big\{ \sup_{s \in [0,t]}\|J_s\|^p\big\}$, due to $(\ref{2-1-1000}),$ there exists constants  $C(p),C(p,T)$ such that
\begin{eqnarray*}
f(t)\leq C(p)+C(p,T)\int_0^tf(s)ds,~\forall t\in [0,T].
\end{eqnarray*}
Then, $(\ref{100-11})$ comes from Gronwall inequality and the proof of $(\ref{100-12})$ is similar.

$(\ref{100-20})$ follows by  $(\ref{100-10})$, $(\ref{100-11})$, $(\ref{100-12})$ and  $(\ref{2-2})$.
\end{proof}

\begin{lemma}\label{3-25}
Let the   Hypothesis $\ref{100-1}$ hold, then for any $T,p>0$,
\label{100-15}
\begin{eqnarray}
\label{3-21}&&\sup_{|(x,y)|\leq R}\sup_{r\in[0,T]}\mathbb{E}_{x,y}\Big\{ \sup_{ t \in [r,T]}\|D_r(x_t,y_t)\|^p \Big\}<\infty,\\
\label{3-22}&&\sup_{|(x,y)|\leq R}\sup_{r \in [0,T]}\mathbb{E}_{x,y}\Big\{\sup_{t\in [r,T]}\big\|D_r J_t^{-1}\big\|^p\Big\}<\infty,\\
\label{3-23}
&&\sup_{|(x,y)|\leq R}\sup_{r \in [0,T]}\mathbb{E}_{x,y}\Big\{\sup_{t\in [r,T]}\big\|D_r J_t\big\|^p\Big\}<\infty,\\
\label{3-24}&& \sup_{|(x,y)|\leq R}\sup_{r_1,r_2\in [0,T]}\mathbb{E}_{x,y}\Big\{\sup_{r_1 \vee r_2\leq t\leq T}\|D_{r_1,r_2}X(t)\|^p\Big\}<\infty.
\end{eqnarray}
\end{lemma}
\begin{proof}
$(\ref{3-21})$, $(\ref{3-24})$ are given
 in Theorem 2.2.1, Theorem 2,2,2,  \cite{Nualart}.
The other two estimations are similar.
\end{proof}

\begin{lemma}\label{100-16} Let  the Hypothesis $\ref{100-1}$ hold, then for any $p,T>0,$  there exists a constant $C(T,p,R)$ such that
\begin{align*}
\sup_{|(x,y)|\leq R}\mathbb{E}_{x,y} \big| \det({M}^{-1}_T)\big|^p\leq C(p,R,T)<\infty.
\end{align*}
\end{lemma}
\begin{proof}
From $(\ref{2-3-2})$ and $(\ref{2-3-1})$,
it only need  to prove   for any $p>0$,
there exists a constant  $C(p,R,T)$ such that
\begin{eqnarray}\label{3-1}
\sup_{|(x,y)|\leq R}\sup_{|v|=1}\mathbb{P}_{x,y}\{v^{*}\tilde{M}_{T}v\leq \epsilon \}\leq C(p,R,T)\epsilon
^{p},\ \ \ \forall \epsilon >0.
\end{eqnarray}
All the    constants appeared in the proof
  of Theorem $\ref{15}$,  we  can  choose them  depending on $R$ but  independent of the  $(x,y)\in \overline{B}(0,R)$  if the Hypothesis $\ref{100-1}$  holds.
So, $(\ref{3-1})$ holds.
In the following paragraphs, we will list  the changes  in   the proof
  of Theorem $\ref{15}$ when we prove this Lemma.

 (1) $R_1$ in $(\ref{2-2-1}), (\ref{2-122})$, $c$ in  $(\ref{2-2-1})$ and  Lemma $\ref{2-1-6}$. Define
   \begin{eqnarray*}
 \Lambda(x,y):=\inf_{|v|=1}\Big(\sum_{j=1}^{j_0}\sum_{V\in\mathcal{A}_j}(v^*V(x,y)V^*(x,y)v)\Big).
 \end{eqnarray*}
 For any $(x,y)\in \overline{B}(0,R)$, $\Lambda(x,y)>0.$
And also for  $ \Lambda(x,y)$ is continuous w.r.t $(x,y)$ (the reason is the same as that  in  $(\ref{2-6})$),  there exists a constant $R_1$ such that
 \begin{eqnarray*}
 \inf_{|(x,y)|\leq R+R_1} \Lambda(x,y)>c:=\frac{1}{2} \inf_{|(x,y)|\leq R} \Lambda(x,y)>0.
 \end{eqnarray*}
If  we choose  $c$ and $R_1$ as above,  we can prove that  the
following inequality holds,
\begin{eqnarray*}
\sum_{j=1}^{j_0}\sum_{V\in \mathcal{A}_{j}}\big(v^{* }V(x',y')\big)^{2}\geq c, \ \forall  (x',y')\in \overline{B}((x,y), R_1), ~(x,y)\in \overline{B}(0,R),~|v|=1.
\end{eqnarray*}

(2)  $R_3$ in $(\ref{2-137})$,   $C_0$ in Lemma $\ref{19}$  and  Lemma $\ref{29}$, set
\begin{eqnarray*}
R_3=\frac{1}{C_0}:=\frac{1}{2} \inf_{(x,y)\in \overline{B}(0,R)} \inf_{|v|=1}\left(v^*b(x,y)b^*(x,y)v\right)>0.
\end{eqnarray*}
From the choosing of $R_3$ and the definition of $\tau'$,  if the Hypothesis $\ref{100-1}$ holds,
 then for any  $s \leq  \tau'$  and process $(x_s,y_s)$ with initial value $(x,y)\in \overline{B}(0,R)$,  the following inequality holds,
\begin{eqnarray*}
\begin{split}
 &|v|^2 \leq C_0\sum_{j=1}^d|v^{*} b_j(x_s,y_s)|^2,  \ \forall v\in  \mathbb{R}^n.
 \end{split}
\end{eqnarray*}
From the above fact and $\tau \leq \tau'$,   we can also  choose  the  following constants
\begin{eqnarray*}
 && C  \  \text{in  Lemma}\  \ref{20};
  \ \ \epsilon_0(q,x,y) \ \text{in Lemma} \ \ref{21};
   \\  &&C(p,T,x,y) \ \text{in}\  (\ref{2-10})    \text{ and  Lemma} \ \ref{21},
 \end{eqnarray*} depending on $R$, but independent of  the  $(x,y)\in \overline{B}(0,R)$.

 (3) The estimate of $\tau$ in Lemma $\ref{2-290}$.  From   Lemma $\ref{100-14}$, for these constants appeared in the proof of   Lemma  $\ref{22}$ and  Lemma   $\ref{2-290}$,  we can choose  them depending on $R$ but  independent of the  $(x,y)\in \overline{B}(0,R)$.

(3) The using   of  Lemma $\ref{23}$ and Lemma $\ref{25}$   in Lemma $\ref{20}$ and   Lemma $\ref{21}$. For example,  in  Lemma $\ref{20}$,  we need to estimate the following probability for some constant $C(R)$,
\begin{eqnarray*}
\mathbb{P}\left\{\int_{0}^{\tau }\|\tilde{y}(s)\|^{2}ds\leq (1+T^{2}C(R))\epsilon ^{q^{3j_{0}+6}},\int_{0}^{\tau }|(v_{1}^{* }A_{s}+v_{2}^{* }C_{s})\nabla
_{y}a_{1}|^{2}ds\geq \epsilon ^{q^{3j_{0}}}\right\},
\end{eqnarray*}
here
\begin{eqnarray*}
d\tilde{y}(s)=-(v_{1}^{* }A_{s}+v_{2}^{* }C_{s})\nabla
_{y}a_{1}ds-\sum_{j=1}^{d}(v_{1}^{* }B_{s}+v_{2}^{* }D_{s})\nabla _{y}b_{j}\cdot
dW_{j}(s).
\end{eqnarray*}
We need  to  check the condition $(\ref{2-126})$  when using    Lemma $\ref{23}$.
Assume  $d(v_{1}^{* }A_{s}+v_{2}^{* }C_{s})\nabla
_{y}a_{1}(x_s,y_s) =K_1(s)ds $ $  +\sum_{j=1}^d K_{2j}(s)dW_j(s)$. From   Lemma $\ref{100-14}$  and the fact
 when  $s\leq \tau $, $|(x,y)|\leq R$  and  $|(x_s,y_s)|\leq R+R_1$,   for any $p>0,$
  there exists  a constant $C=C(T,p,R)$  such that,
\begin{eqnarray*}
  &&\mathbb{E}_{x,y}\sup_{0\leq t \leq \tau} \bigg(| K_1(s)|^p+|K_2(s)|^p+|(v_{1}^{* }A_{s}+v_{2}^{* }C_{s})\nabla
_{y}a_{1}(x_s,y_s)|
+\sum_{j=1}^d|(v_{1}^{* }B_{s}+v_{2}^{* }D_{s})\nabla _{y}b_j|\bigg)^p
\\&&\leq C(p,R,T)<\infty.
\end{eqnarray*}

(4) For the other constants appeared in the  proof of Theorem $\ref{15}$,  we  can  also choose them  depending on $R$ but  independent of the $(x,y)\in \overline{B}(0,R)$.
\end{proof}

We are now in a position to give
\begin{proof}
{\bf\text The proof of Theorem $\ref{2-39}$:}
For  any  $\xi \in \mathbb{R}^{m+n}$,
\begin{eqnarray*}
  \langle \nabla P_{t}f\left( x,y\right) ,\xi \rangle &=&\mathbb{E}_{x,y}\nabla
f\left(
x_t,y_t \right) J_t \xi.
\end{eqnarray*}
Assume $x_t=(x_t^1,\cdots,x_t^m)$ and $y_t=(y_t^1,\cdots,y_t^n)$, then  from $(2.29),(2.30)$  in \cite{Nualart},
\begin{eqnarray*}
&& \mathbb{E}_{x,y}\Big\{\nabla_{i}
f\left(
x_t,y_t \right) J_t \xi\Big\}
\\ && =\sum_{k=1}^{m}\mathbb{E}\Big\{f(x_t,y_t)\delta\big(J_t \xi (M_t^{-1})^{i,k}x_t^{k}\big)\Big\}+\sum_{k=m+1}^{m+n}\mathbb{E}\Big\{f(x_t,y_t)\delta\big(J_t \xi (M_t^{-1})^{i,k}y_t^{k-m}\big)\Big\}.
\end{eqnarray*}
So, this Theorem  comes from  Lemma $\ref{100-14}$,  Lemma  $\ref{3-25}$,
Lemma $\ref{100-16}$ and   Proposition 1.5.8  in \cite{Nualart}.
\end{proof}

In the end of this section, we give a Proposition which  is supplementary to this article.
\begin{proposition}\label{2-35}
Let  $a_1,a_2,b \in C_b^2$ and the Hypothesis $\ref{2-18}$ hold,  then the law of
$(x_t,y_t)$  with  initial value $(x,y)$    is absolutely
continuous with respect to Lebesgue measure and its density function
$p(t,(u,v))$ is continuous w.r.t $(u,v) \in  \mathbb{R}^m\times \mathbb{R}^n$ for fixed t.
Furthermore, the following estimation holds
 \[
 \sup_{(u,v)\in \mathbb{R}^m \times \mathbb{R}^n}|p(t,(u,v))| <\infty.
 \]
\end{proposition}
\begin{proof}
  It directly comes from  the   Theorem 5.9 in \cite{Shigekawa} and the   Theorem $\ref{15}$  in this article.
\end{proof}

\section{Strong Feller Property}
In this section, we prove that the semigroup $P_t$  associated with Eq.$(\ref{1-1})$  is strong Feller under some conditions.
 From Theorem  $\ref{2-39}$, $P_t$ is strong Feller under some conditions which need   all the coefficients for Eq.$(\ref{1-1})$ are in  $C_b^2.$
But in the Hamiltonian  systems, the diffusion and drift part are
polynomial growth, so the Theorem  $\ref{2-39}$ can't apply directly.
But if the SDE has global solution, we can also prove $P_t$ is
strong Feller without the bounded conditions.

The followings are our Hypothesis and Theorem in this section.
\begin{hypothesis}
\label{3-20} There exists $j_0 \in \mathbb{N}$ such that:
\begin{description}
  \item[(\romannumeral1)]
  $a_1 \in C^{j_0+2}(\mathbb{R}^m\times \mathbb{R}^n;\mathbb{R}^m), ~a_2
\in C^2(\mathbb{R}^m\times \mathbb{R}^n;\mathbb{R}^n);$
  \item[(\romannumeral2)]
  $b \in
C^2(\mathbb{R}^m\times
\mathbb{R}^n;\mathbb{R}^n\times \mathbb{R}^d )$,~
$\det(b(x,y)\cdot b^{*}(x,y))\neq 0,~\forall (x,y) \in \mathbb{R}^m \times \mathbb{R}^n;$
  \item[(\romannumeral3)]
  $\forall (x,y) \in
\mathbb{R}^{m}\times \mathbb{R}^{n},$  the vector space spanned by  $\cup_{k=1}^{j_0} \mathcal{A}_{k}$ at  point $(x,y)$ has dimension $m;$
  \item[(\romannumeral4)]
  The solution to  equation $(\ref{1-1})$  globally exists  for any initial value $(x,y) \in \mathbb{R}^{m}\times \mathbb{R}^{n} .$
\end{description}
\end{hypothesis}
\begin{Rem}
If there exists a Liapunov function $W$ such that $LW\leq cW$ for some $c>0$, then
the  $\bf{(\romannumeral4)}$ in Hypothesis $\ref{3-20}$ holds by Theorem  5.9,
\cite{Bellet}. Here
\begin{align*}
  L=\sum_{i=1}^ma_1^{i}\frac{\partial }{\partial x_i}+\sum_{i=1}^na_2^i\frac{\partial }{\partial y_i}+
  \frac{1}{2}\sum_{i,j=1}^n(b\cdot b^{*})_{i,j}\frac{\partial^2}{\partial  y_i \partial y_j}.
\end{align*}

\end{Rem}
\begin{thm} \label{4-16}
Let  the Hypothesis $\ref{3-20}$ hold, then~$ P_t$~ is  strong
Feller.
\end{thm}

For the convenience of writing, we will use $x$ instead of $(x,y)$ in the rest of this section.
Let $X_t^x=(x_t,y_t)$ be the solution of $(\ref{1-1})$ with initial
value $x\in \mathbb{R}^m\times  \mathbb{R}^n$. In the following part,
we would like to use the localization to prove Theorem $\ref{4-16}$.

For any fixed $l\in \mathbb{N}, $ set  $a(x)=(a_1^{*}(x),a_2^{*}(x))^{*},$ $g_{l}(x)=h_l(x)a(x),$ $q_l(x)=h_l(x) b(x), h_l(x)\in \mathbb{R}$  is a smooth function with compact support  and  $h_l(x)=1$\ on $B^{\circ}(0,l).$
Let $X^l_s(x)$\ be the solution to the following equation,
\begin{eqnarray}
\label{3-35}
\begin{split}
&\ \ \
X^l_s(x)=x+\int_0^s g_{l}(X^l_r(x))dr
 +\int_0^s\left(
\begin{matrix}
 0 \\
 q_l(X^l_r(x))%
\end{matrix}%
\right) dW_{r}.
\end{split}
\end{eqnarray}
Define a sequence of stopping time
\begin{eqnarray*}
S_l(x)&=&\inf\{s>0,~X^l_s(x)\not\in B^{\circ}(0,l)\},\ l\geq1.
\end{eqnarray*}
If the Hypothesis $\ref{3-20}$ holds,    then for any $x\in\mathbb{R}^{m+n},$ the following properties holds  a.s.
\begin{align}
\label{3-3}
 & S_l(x)<S_{l+1}(x),\\
\label{3-4}& X^l_s(x)=X^{l+1}_s(x),~  \forall  s \in  \text{$[0,S_l(x))$},
\\ \label{3-5} & X_s^x=X^l_s(x),~~~  \ \  \forall  s \in \text{$[0,S_l(x))$},
\\ \label{3-6}& \sup_{l}S_l(x)=\infty.
\end{align}

In order to prove  Theorem $\ref{4-16}$, we also   need the following  Lemmas.
\begin{lemma}\label{3-70}
Let  the Hypothesis $\ref{3-20}$ hold, then for any $x_0 \in \mathbb{R}^{m+n},$ $l\geq 2,~t>0$
$$
\limsup_{x\rightarrow x_0} I_{\{t>S_l(x)\}}\leq I_{\{t\geq
S_{l-1}(x_0)\}},~ a.s.
$$
\end{lemma}

\begin{proof}
Let $\Gamma$ be a measurable set with $\mathbb{P}(\Gamma^c)=0$ and such that $X^{l}_s(x,\omega)$ is continuous  w.r.t.
$s$ and  $x$ for $\omega\in\Gamma$. For $\omega \in \Gamma$, the conclusion is apparent if
\begin{eqnarray*}
\limsup_{x\rightarrow x_0}I_{\{t>S_l(x)\}}(\omega)=0\text{ or }I_{\{t\geq S_{l-1}(x_0)\}}(\omega)=1.
\end{eqnarray*}
Assume that $\limsup\limits_{x\rightarrow x_0}I_{\{t>S_l(x)\}}(\omega)=1$ and $I_{\{t\geq S_{l-1}(x_0)\}}(\omega)=0$,
then
\begin{align}
\sup_{s \in [0,t]} |X^{l-1}_s(x_0,\omega)| \leq  l-1.
\end{align}
Furthermore, by $(\ref{3-4})$,
\begin{eqnarray*}
\sup_{s \in [0,t]} |X^{l}_s(x_0,\omega)| \leq  l-1.
\end{eqnarray*}
Since $\limsup\limits_{x\rightarrow x_0}I_{\{t>S_l(x)\}}(\omega)=1$, then there exist $\{x_n\}\subset\mathbb{R}^{m+n}$ with
$x_n\rightarrow x_0$ as $n\rightarrow\infty$, such that for $n$ large enough
\begin{align}\label{5-2}
\sup_{s \in [0,t]}|X^{l}_s(x_{n},\omega)|\geq l.
\end{align}
From  $t<S_{l-1}(x_0)$,  $(\ref{3-4})$  and  $\sup_{s \in [0,t]}
|X^{l-1}_s(x_0,\omega)|\leq  l-1$,
\begin{align}\label{3-50}
\sup_{s \in [0,t]} |X^{l}_s(x_0,\omega)| \leq  l-1.
\end{align}

For  $X^{l}_s(x,\omega)$ is continuous  w.r.t.
$s$ and  $x$  and $[0,t]\times \overline{B}(0,1)\subseteq [0,\infty)\times \mathbb{R}^{m+n}$ is a compact set, so for $\epsilon_0=\frac{1}{2}$ there exists  $\delta_0>0$ such that for any $|x-x_0|\leq \delta_0$ and $s\in[0,t]$
\begin{align*}
|X^{l}_s(x_0,\omega)-X^{l}_s(x,\omega)|\leq \frac{1}{2}.
\end{align*}
That means when  $|x-x_0|\leq \delta_0$,
\begin{align*}
  \sup_{s\in [0,t]}|X^{l}_s(x,\omega)|\leq  \sup_{s \in [0,t]} |X^{l}_s(x_0,\omega)| +\frac{1}{2},
\end{align*}
this contradict $(\ref{5-2})$ and $(\ref{3-50})$.
\end{proof}

Let $\{P^l_t\}_{t\geq0}$ be the transition semigroup of $(\ref{3-35})$.
\begin{lemma}\label{3-92}
Let  the Hypothesis $\ref{3-20}$ hold, then  for any $f\in\mathscr{B}_b(\mathbb{R}^{m+n};\mathbb{R})$,  $P^l_tf$ is continuous on $B^{o}(0,l)$.
\end{lemma}
\begin{proof}
For any $0<l_0<l,$ since $b_l=b$ on $\overline{B}(0,l)$, then
 \begin{align*}
 \inf_{y \in \overline{B}(0,l_0)}\inf_{|a|=1}\left(|aA|^2+|a\nabla b_l(y)A|^2\right)>0.
\end{align*}
In the Hypothesis $\ref{100-1}$ and the  proof of Theorem $\ref{2-39}$, let $R=l_0$ and substitute $b_l$ for $b$, then
\begin{eqnarray*}\label{3-2}
 |\nabla P_t^lf(y)|\leq C(l_0,t)\|f\|_\infty, \ \  \forall f\in C_b,\forall y\in\overline{B}(0,l_0).
\end{eqnarray*}
\end{proof}

We are now in a position to give
\begin{proof} {\bf\text The proof of Theorem $\ref{4-16}$:}
For $f\in \mathscr{B}_{b}(\mathbb{R}^{m+n})$ with $f\geq 0$   and  $x_0 \in \mathbb{R}^{m+n}$ with $|x_0|<l$
\begin{eqnarray}
\nonumber
&&\limsup_{x\rightarrow x_0}\mathbb{E}f(X_t^x)
\\ \nonumber &&\ \ \leq \limsup_{x\rightarrow x_0}\mathbb{E}\big\{f(X_t^x)I_{\{t\leq S_l(x)\}}\big\}+\limsup_{x\rightarrow x_0}
\mathbb{E}\big\{f(X_t^x)I_{\{t>S_l(x)\}}\big\}
\\
\nonumber
&&\ \ \leq \limsup_{x\rightarrow x_0}\mathbb{E}\big\{f(X^l_t(x))I_{\{ t\leq S_l(x)\}}\big\}
+\|f\|_{\infty}\limsup_{x\rightarrow x_0}\mathbb{P}(t>S_l(x))
\\
\nonumber
&&\ \ \leq \limsup_{x\rightarrow x_0}\mathbb{E}f(X^{l}_t(x))
+\|f\|_{\infty}\limsup_{x\rightarrow x_0}\mathbb{P}(t>S_l(x))
\\ \label{3-36}
&&\ \  = \mathbb{E}f(X^{l}_t(x_0))
+\|f\|_{\infty}\limsup_{x\rightarrow x_0}\mathbb{P}(t>S_l(x)),
\end{eqnarray}
where we use $(\ref{3-5})$ in the second inequality and Lemma $\ref{3-92}$ in the last equality.
It follows $(\ref{3-36})$ and $(\ref{3-5})$ that
\begin{eqnarray*}
\nonumber
&& \limsup_{x\rightarrow x_0}\mathbb{E}f(X_t^x)
\\ && \leq
\mathbb{E}\big\{f(X^{l}_t(x_0))I_{\{t\leq S_l(x_0)\}}\big\}+\mathbb{E}\big\{f(X^l_t(x_0))I_{\{t>S_l(x_0)\}}\big\}
+\|f\|_{\infty}\limsup_{x\rightarrow x_0}\mathbb{P}(t>S_l(x))
\\ \label{3-39}
&&\leq \mathbb{E}f(X_t^{x_0})+\|f\|_{\infty}\mathbb{P}(t>S_l(x_0))+\|f\|_{\infty}\limsup_{x\rightarrow x_0}\mathbb{P}(t>S_l(x)).
\end{eqnarray*}
Let $l\rightarrow \infty$ in the above inequality  and by Lemma $\ref{3-70}$ we obtain
\begin{eqnarray*}
\limsup_{x\rightarrow x_0}\mathbb{E}f(X_t^x)&\leq & \mathbb{E}f(X_t^{x_0})+\|f\|_{\infty}\lim_{l\rightarrow \infty}\limsup_{x\rightarrow x_0}\mathbb{P}(t>S_l(x))
\\&\leq & \mathbb{E}f(X_t^{x_0})+\|f\|_{\infty}\lim_{l\rightarrow \infty}\mathbb{E}\limsup_{x\rightarrow x_0}I_{\{t>S_l(x)\}}
\\&\leq&\mathbb{E}f(X_t^{x_0})+\|f\|_{\infty} \lim_{l\rightarrow \infty}\mathbb{E} I_{\{t\geq S_{l-1}(x_0)\}}
\\&\leq & \mathbb{E}f(X_t^{x_0}).
\end{eqnarray*}
For $g\in\mathscr{B}_b(\mathbb{R}^{m+n})$, repeating the above procedure with $\|g\|_\infty-g$ and $\|g\|_\infty+g$, one arrives at
\begin{eqnarray}
\label{3-43}
\limsup_{x\rightarrow x_0}\mathbb{E}\big\{\|g\|_\infty-g(X_t^x)\big\}&\leq & \|g\|_\infty-\mathbb{E}g(X_t^{x_0}),
\\ \label{3-42}\limsup_{x\rightarrow x_0}\mathbb{E}\big\{\|g\|_\infty+g(X_t^x)\big\}&\leq & \|g\|_\infty+\mathbb{E}g(X_t^{x_0}).
\end{eqnarray}
Therefore, for any  $g\in\mathscr{B}_b(\mathbb{R}^{m+n})$
\begin{eqnarray*}
\label{3-41}
\lim_{x\rightarrow x_0}\mathbb{E}g(X_t^x)&= & \mathbb{E}g(X_t^{x_0}).
\end{eqnarray*}
\end{proof}

\begin{Rem}
In   \cite{Delarue},   the authors considered   the following SDE
\begin{eqnarray*}
\left\{
  \begin{split}
X_t^1&=x_1+\int_0^t F_1(s,X_s^1,\cdots,X_s^n)ds+\int_0^t\sigma(s,X_s^1,\cdots,X_s^n)dW_s,
\\
X_t^2&=x_2+\int_0^t F_2(s,X_s^1,\cdots,X_s^n)ds,
\\
X_t^3&=x_3+\int_0^t F_3(s,X_s^2,\cdots,X_s^n)ds,
\\
&  \ \ \vdots
\\
X_t^{n}&=x_n+\int_0^tF_n(s,X_s^{n-1},X_s^n)dt.
\end{split}
\right.
\end{eqnarray*}
 the authors proved  that  $X_t$ has a density $p(t,x,y)$ and gave the upper and lower bounds of $p(t,x,y)$ if the spectrum of the matrix-valued function $ A= \sigma\cdot\sigma^{*}$ is
included in $[\Lambda^{-1},\Lambda]$ for some $\Lambda\geq 1$.
In our article,  we can't obtain such strong results since in our condition is $\det\big(\sigma(x)\cdot\sigma^{*}(x)\big)\neq 0,$  which is weaker than that in  \cite{Delarue}.
\end{Rem}

\section{Some Applications}
The strong Feller property is very useful  when we   prove  the uniqueness
of invariant measure.

If $ X_t \in \mathbb{R}^n, ~t \in  [0,+\infty),~n\in \mathbb{N} $ is
a  continuous  Markov process. The following  theorem is classical.
\begin{hypothesis}Let  $P_t$ be the  semigroup associated  with $X_t$, and
\label{4-13}
\begin{itemize}
\item the Markov process $X_t$ is irreducible, i.e,
$$
P_{t}(x,A)>0, \ {\text for\ all }\ t>0,\ x \in \mathbb{R}^n, {\text
\
  open\  set } \ A.
$$

\item   $P_t $ is strong Feller.
\end{itemize}
\end{hypothesis}

\begin{thm}
\label{4-4} (c.f. \cite{Stettner} \cite{Flandoli2}
\cite{Flandoli}) Let  the   Hypothesis $\ref{4-13}$  hold, then $P_t $
exists at most one invariant measure.
\end{thm}
\subsection{The Langevin Equation}
This example is extended from the one in  \cite{J.C. Mattingly}.
Let $W_t,t\ge 0$ be a standard d-dimensional
Brownian Motion and ~$F:\mathbb{R}^{d}\rightarrow R,~\sigma \in
\mathbb{R}^{d \times d}$~invertible. Consider the Langevin SDE for
$q,p \in \mathbb{R}^{d}$ the position and momenta of particle of
unit mass, namely
\begin{eqnarray}
\label{4-1}
\left\{
\begin{split}
&dq=pdt,\\
&dp=-\gamma pdt-\nabla F(q)dt +\sigma dW_t.
\end{split}\right.
\end{eqnarray}

\begin{hypothesis}
\label{4-5} The function $F \in C^3(\mathbb{R}^d,\mathbb{R})$~and
satisfy
\begin{itemize}
\item  $F(q) \geq 0 $ for all $q  \in \mathbb{R}^d;$
\item There exists an $\alpha >0$ and\ $ \beta  \in(0,1) $ such that
$$
\frac{1}{2} \langle \nabla F(q),q \rangle \geq \beta F(q)+\gamma^2 \frac{\beta (2-\beta)}{8(1-\beta)}||q||^2-\alpha.
$$
\end{itemize}
\end{hypothesis}

\begin{proposition}
Let the  Hypothesis $\ref{4-5}$ hold,
then the
semigroup $P_t$ associated with the Langevin SDE  is strong Feller and
has a unique  invariant measure.
\end{proposition}

\begin{proof}
First, the Hypothesis $\ref{3-20}$ holds   for $j_0=1$, so
$P_t$ is strong Feller by Theorem $\ref{4-16}$.  Second, by the
Lemma 3.4 in \cite{J.C. Mattingly}, we know that  $P_t$ is
irreducible. So  $P_t$  has  at most one invariant measure. Third,
by the Corollary A.5 in \cite{J.C. Mattingly}, the invariant measure
for $P_t$ exists.
\end{proof}

\subsection{Stochastic Hamiltonian Systems}

This  example  is extended from the one in \cite{D.Talay}.  Consider a stochastic differential system of the type
\begin{eqnarray}
\label{4-11}
\left\{
\begin{split}
X_t&=X_0+\int_0^t \partial_y H(X_s,Y_s)ds,\\
Y_t&=Y_0-\int_0^t\big[\partial_x H(X_s,Y_s)+
F(X_s,Y_s)\partial_yH(X_s,Y_s)\big]ds+W_t,
\end{split}\right.
\end{eqnarray}
where  $X_t,Y_t,W_t$ belong to $\mathbb{R}^d$.

In the following Hypothesis, we don't  need $F$ and $H \in
C^{\infty}$ as in \cite{D.Talay}.
\begin{hypothesis}
\label{4-8} There exists  strictly positive numbers $\nu,M,\delta$,  there exits  a function $R(x,y)$ on\ $ \mathbb{R}^{2d } $~with
second derivatives having polynomial growth at infinity, such that
\begin{itemize}
\item $F\in C^2,H \in C^4$;
\item  $ 0< \nu |\xi|^2 \leq \sum_{i,j=1}^d  \partial_{y_iy_j}H(x,y) \xi_i \xi_j,$ $\forall x,y,\xi$;
\item  $H(x,y)+R(x,y)+M \geq \delta (|x|^{\nu}+|y|^{\nu})$;
\item  $LH(x,y)+LR(x,y) \leq -\delta (H(x,y)+R(x,y))+M$;
\item   $|\partial_yH(x,y)+\partial_yR(x,y)|^2 \leq M(H(x,y)+R(x,y)+1)$.
\end{itemize}
\end{hypothesis}

\begin{proposition}
Let the   Hypothesis  $\ref{4-8}$ holds, then
the semigroup $P_t$ associated with the equation $(\ref{4-11})$
is strong Feller and  has a unique  invariant measure.
\end{proposition}

\begin{proof}
First, for $ 0< \nu |\xi|^2 \leq \sum_{i,j=1}^d
\partial_{y_iy_j}H(x,y) \xi_i \xi_j,$ $\forall x,y,\xi$, we have the
the Hypothesis $\ref{3-20}$ is satisfied for $j_0=1$. Thus $P_t$ is
strong feller by Theorem $\ref{4-16}$. Second, by the Lemma 2.2 in
\cite{D.Talay}, We know the $P_t$ is irreducible. So the invariant
for $P_t$ is at most one. Third, by the Lemma 2.1 and Corollary 2.1
in \cite{D.Talay},   the invariant measure for $P_t$ exists.
\end{proof}
\subsection{ High Order Stochastic Differential Equations }
Consider the following Stochastic Differential Equations  with order
$n$,
\begin{eqnarray} \label{4-17}
x^{(n)}_t=f(x^{(n-1)}_t,\cdots,x_t)+b(x^{(n-1)}_t,\cdots,x_t)\dot{B}_t,
\end{eqnarray}
 where $x^{(k)}_t=\frac{d^{k}x_t}{dt^{k}},~k=1,\cdots,n, ~x_t \in \mathbb{R}^m,~b \in \mathbb{R}^{m\times d},~B_t \in \mathbb{R}^d$.

Set $y_i(t)=x^{(i-1)}_t,~ 1 \leq i\leq n$, then
$y_t=(y_1(t),\cdots,y_n(t))$ satisfy the following stochastic
differential equation:
\begin{eqnarray}\label{4-15}
\left\{\begin{array}{l} dy_1(t)=y_2(t)dt,
\\ \vdots
\\dy_{n-1}(t)=y_{n}(t)dt,
\\dy_n(t)=f(y_n,y_{n-1},\cdots,y_1)dt+b(y_n,y_{n-1}\cdots,y_1)dB_t.
\end{array}
\right.
\end{eqnarray}

\begin{proposition} \label{4-18}
Let $x^x_t$ be the solution of equation $(\ref{4-17})$~with initial
value $x=(x_0,\cdots,x_0^{(n-1)})\in \mathbb{R}^{m\times n},$  $P_t$ be  the semigroup associated with $(\ref{4-17})$,
\begin{enumerate}
\item[$(1)$] If $f \in C_b^2(\mathbb{R}^{m\times n};\mathbb{R}^{m}),
b \in C_b^2(\mathbb{R}^{m\times n};\mathbb{R}^{m})$ and $\det(b(x)b^{*}(x)) \neq 0$, then the law of $x_t^x$ is
   absolutely continuous with respect to Lebesgue measure, and its
   density $p(t,x,y)$ is continuous with respect to y and  $\sup_y |p(t,x,y)|<\infty.$
\item[$(2)$]  If  $f \in C^2(\mathbb{R}^{m\times n};\mathbb{R}^{m}),
b \in C^2(\mathbb{R}^{m\times n};\mathbb{R}^{m})$  and   for any  $x
\in \mathbb{R}^{m\times n},$ $\det(b(x)b^{*}(x)) \neq 0 $  and   the
solution to  equation $(\ref{4-17})$ with initial value $x$ is
globally exists,  then  the semigroup  $P_t$ is strong Feller.
\end{enumerate}
\end{proposition}

\begin{proof}
The  Hypothesis $\ref{3-20}$  holds  for $j_0=1$, so
$(1)$ follows from  Proposition $\ref{2-35}$.
And $(2)$ follows from  Theorem $\ref{4-16}$,.
\end{proof}

Specially, if we consider the following stochastic differential
equation
\begin{eqnarray}
\label{4-14}
&&x^{(n)}_t+a_{n-1}(x_t)x^{(n-1)}_t+\cdots+a_0(x_t)x_t+c(x_t)+\frac{b(x_t)dB_t}{dt}=0,
\end{eqnarray}
where $x^{(k)}_t=\frac{d^{k}x_t}{dt^k}, ~x_t\in \mathbb{R}^m,~B_t \in
\mathbb{R}^d,  ~b(x_t) \in \mathbb{R}^{m\times d}, ~c\in
\mathbb{R}^m,~a_0,\cdots,a_{n-1}\in \mathbb{R}^{m\times m}$.

\begin{corollary}
Let $x_t^x$ be the solution of equation $(\ref{4-14})$~with initial
value $x=(x_0,\cdots,x_0^{(n-1)})\in \mathbb{R}^{m\times n}.$
\begin{enumerate}
\item[$(1)$] If  $a_0,\cdots,a_{n-1} \in C_b^2(\mathbb{R}^m;\mathbb{R}^{m\times m}),
    ~b \in C_b^2(\mathbb{R}^m;\mathbb{R}^{m\times d}),~c \in
   C_b^2(\mathbb{R}^m;\mathbb{R}^m)$, and  $\det(b(x_0)b^{*}(x_0)) \neq 0$, then the law of $x_t^x$ is
   absolutely continuously with respect to Lebesgue  measure, and  its
   density $p(t,x,y)$ is
   continuous with respect to y and  $\sup_y |p(t,x,y)|<\infty.$
\item[$(2)$]   If  $a_0,\cdots,a_{n-1} \in C^2(\mathbb{R}^m;\mathbb{R}^{m\times m}),~b \in C^2(\mathbb{R}^m;\mathbb{R}^{m\times d}),
    ~c \in  C^2(\mathbb{R}^m;\mathbb{R}^m)$, and for any $x=(x_0,\cdots,x_0^{(n-1)})\in \mathbb{R}^{m\times n}$,  $\det(b(x_0)b^{*}(x_0)) \neq 0$ and   $x_t^x$ is globally
    exists, then the semigroup $P_t$ is strong Feller.
\end{enumerate}
\end{corollary}

\begin{proof}
It can be obtained by Proposition $ \ref{4-18}$.
\end{proof}

\begin{appendices}

\section{Proof of Lemma  $\ref{23}$ and Lemma $\ref{25}$} \label{A-2}
The proof of Lemma  $\ref{23}$ and Lemma $\ref{25}$ are  very similar to the proof of  Norris Lemma (c.f. Lemma 2.3.1, \cite{Nualart}), so we only give the proof of Lemma $\ref{25}$ here.
\begin{proof} {\bf Proof of Lemma  $\ref{25}$:}
Define stopping time as
\begin{eqnarray*}
\zeta=\inf\left\{t\geq 0:\sup_{0\leq s \leq t}(|a(s)|+|u(s)|)>\epsilon ^{-r}\right\}\wedge
\sigma,
\end{eqnarray*}%
then
\begin{eqnarray*}
B=\left\{\int_{0}^{\sigma }\tilde{y}(t)^{2}dt<\epsilon ^{q},\int_{0}^{\sigma
}|u(t)|^{2}dt\geq \epsilon \right\} \subseteq A_{1}\cup A_{2}\cup A_{3},
\end{eqnarray*}%
 where
\begin{eqnarray*}
A_{1} &=&\left\{\int_{0}^{\sigma }\tilde{y}(t)^{2}dt<\epsilon ^{q},~\int_{0}^{\sigma
}|u(t)|^{2}dt\geq \epsilon ,~\zeta=\sigma ,~\sigma \geq \epsilon ^{ }\right\}, \\
A_{2} &=&\left\{\zeta<\sigma\right\}, \\
A_{3} &=&\left\{\sigma <\epsilon ^{ }\right\}.
\end{eqnarray*}%
Obviously,
\begin{eqnarray*}
\mathbb{P}\{A_{2}\} \leq \tilde{c}\epsilon ^{rp},  \
\mathbb{P}\{A_{3}\} \leq C(c_\sigma,\tilde{p})\epsilon ^{
\tilde{p}},
\end{eqnarray*}
so we only need to estimate $\mathbb{P}(A_1)$.

Introduce the following notation
\begin{eqnarray*}
&&N_{t}=\sum_{i=1}^d\int_{0}^{t}\tilde{y}(s)u_{i}(s)dW_{i}(s),
\\ && M_t=\sum_{i=1}^d \int_0^t u_i(s)dW_i(s),\\
&& B=\left\{\langle N\rangle_{\sigma }< \rho
_{1},\sup_{0\leq s\leq \sigma }|N_{s}|\geq \delta _{1}\right\},
\end{eqnarray*}%
where $\rho _{1}=\epsilon ^{q-2r},~\delta _{1}=\epsilon ^{\frac{q}{2}-r-\frac{%
v}{4}}.$

We will prove that there exists $\epsilon _{0}=\epsilon
_{0}(c_\sigma,q,r,v),$ such that
\begin{eqnarray*}
A_{1}\subseteq B,~\text{for all }\epsilon \leq \epsilon _{0}.
\end{eqnarray*}%
If this has been proved, then
\begin{eqnarray*}
\mathbb{P}\{A_{1}\}\leq \mathbb{P}\{B\}\leq 2\exp \{-\frac{\delta _{1}^{2}}{2\rho _{1}}\}\leq
\exp \{-\epsilon ^{-\frac{v}{4}}\},
\end{eqnarray*}%
thus this Lemma holds.

In the below, we will to prove: there exists $\epsilon _{0}=\epsilon
_{0}(c_\sigma,q,r,v ),$ such that
\begin{eqnarray*}
A_{1}\subseteq B,~\text{for all }\epsilon \leq \epsilon _{0}.
\end{eqnarray*}
Set $\epsilon _{0}=\epsilon _{0}(c_\sigma,q,r,v),$ such that for $%
\epsilon \leq \epsilon _{0}(c_\sigma,q,r,v)$, the following
inequalities hold.
\begin{align*}
& \epsilon ^{q}+2c_\sigma(\sqrt{c_\sigma}\epsilon
^{\frac{q}{2}-r}+\delta
_{1})\leq \epsilon ^{\frac{q}{2}-r-\frac{v}{4}}(1+2c_\sigma), \\
& \epsilon ^{\frac{q}{4}-\frac{r}{2}-\frac{v}{8}}(1+2c_\sigma)+\epsilon ^{%
\frac{5q}{4}-\frac{5r}{2}-\frac{v}{8}}<\epsilon .
\end{align*}
We  only need to prove for  any $\epsilon \leq \epsilon
_{0}(c_\sigma,q,r,v)$, $\omega\in B^{c}$ implies $ \omega \in
A_{1}^{c}.$

Let $\omega \in B^{c},~\int_{0}^{\sigma }\tilde{y}(t)^{2}dt<\epsilon ^{q},~\sigma(\omega)=\zeta(\omega) \geq \epsilon$, then
similar to
the estimations of $\sup\limits_{0\leq s\leq T}\Big|\int_0^tY_sdY_s\Big|,~\int_0^T\langle M\rangle_tdt $  and $\langle M\rangle_T $ in  the proof of Lemma 2.3.2 in \cite{Nualart}, we can  obtain
\begin{align}
\nonumber &\sup_{t\leq \sigma }\Big|\int_{0}^{t}\tilde{y}_{s}d\tilde{y}_{s}\Big|\leq \sqrt{c_\sigma}\epsilon ^{\frac{q}{2}-r}+\delta _{1},
\\ \nonumber  &\int_{0}^{\sigma}\langle M\rangle _{t}dt\leq \epsilon ^{\frac{q}{2}-r-\frac{v}{4}}(1+2c_\sigma),
\\ \label{A-3} &\langle M\rangle _{\sigma } \leq  \gamma^{-1}\epsilon
^{\frac{q}{2}-r-\frac{v}{4}}(1+2c_\sigma)+\gamma\epsilon ^{-2r},\ \  \forall \gamma\in (0,\sigma) .
\end{align}
Choosing  $\gamma=\epsilon ^{\frac{1}{2}\left( \frac{q}{2}-r-\frac{v}{4}\right) }<\epsilon
^{ }\leq \sigma$ in $(\ref{A-3}).$
Since $2q>8+20r+v$, we have $\langle M \rangle _{\sigma }<\epsilon $,  i.e.  $%
\omega \in A_{1}^{c}$.
\end{proof}

\end{appendices}

\section*{Acknowledgements}
The authors thanks for Professor  Fuzhou Gong,  Zhiming Ma, Pei Xu,
Tusheng  Zhang and Xicheng Zhang  and  Dr. Pei Yan Li for their
valuable discussions. Especially, we also thanks for Dr. Jian Zhou
for his much valuable discussions on his doctor period.

\end{document}